\documentclass{article}

\usepackage{arxiv}

\usepackage[utf8]{inputenc} 
\usepackage[T1]{fontenc}    
\usepackage{hyperref}       
\usepackage{url}            
\usepackage{booktabs}       
\usepackage{amsfonts}       
\usepackage{nicefrac}       
\usepackage{microtype}      
\usepackage{xcolor}         
\usepackage[square,numbers]{natbib}
\usepackage{doi}

\usepackage{amsmath}
\usepackage{algorithm}
\usepackage{algpseudocode}
\usepackage{amsthm}
\usepackage{graphicx}     
\usepackage{subcaption}
\usepackage{amssymb}

\theoremstyle{definition}

\newtheorem{theorem}{Theorem}
\newtheorem{proposition}{Proposition}
\newtheorem{lemma}{Lemma}

\newtheorem{remark}{Remark}
\newtheorem{corollary}{Corollary}

\def\bs{\boldsymbol}
\def\E{\mathbb{E}}
\DeclareMathOperator*{\tr}{Tr}

\title{Sampling from the Random Linear Model via Stochastic Localization Up to the AMP Threshold}

\author{
 Han Cui \\
  Department of Statistics\\
  University of Illinois Urbana-Champaign \\
  Champaign, IL 61820 \\
  \texttt{hancui5@illinois.edu} \\
  \And
  Zhiyuan Yu \\
  Department of Statistics\\
  University of Illinois Urbana-Champaign \\
  Champaign, IL 61820 \\
  \texttt{yu124@illinois.edu} \\
  \AND
  Jingbo Liu \\
  Department of Statistics\\
  University of Illinois Urbana-Champaign \\
  Champaign, IL 61820 \\
  \texttt{jingbol@illinois.edu} \\
}

\renewcommand{\shorttitle}{Sampling from the Random Linear Model}

\begin{document}
\maketitle

\begin{abstract}
  Recently, Approximate Message Passing (AMP) has been integrated with stochastic localization (diffusion model) by providing a computationally efficient estimator of the posterior mean.
Existing (rigorous) analysis typically proves the success of sampling for sufficiently small noise,
but determining the exact threshold involves several challenges.
In this paper, we focus on sampling from the posterior in the linear inverse problem, with an i.i.d.\ random design matrix,
and show that the threshold for sampling coincides with that of posterior mean estimation. 
We give a proof for the convergence in smoothed KL divergence
whenever the noise variance $\Delta$ is below $\Delta_{\rm AMP}$,
which is the computation threshold for mean estimation introduced in (Barbier et al.,
2020).
We also show convergence in the Wasserstein distance under the same threshold assuming a dimension-free bound on the operator norm of the posterior covariance matrix, 
a condition strongly suggested by recent breakthroughs on operator norm bounds in similar replica symmetric systems.
A key observation in our analysis is that phase transition does not occur along the sampling and interpolation paths assuming $\Delta<\Delta_{\rm AMP}$.
\end{abstract}

\section{INTRODUCTION}

Linear inverse problem is fundamental to a wide range of fields in science, engineering, and technology, including medical imaging \citep{qu2013magnetic, song2022solving, marinus2010magnetic}, compressed sensing \citep{bayati2011dynamics, bayati2012thelasso, david2009message} and image restoration \citep{gonzalez2006digital}.
In this paper, we consider the linear inverse problem of reconstructing a signal $\bs{\theta}\in\mathbb{R}^N$ from noisy measurements $\bs{y}\in\mathbb{R}^M$ in Gaussian random linear model
\begin{align}
    \bs{y} = \bs{\phi\theta}+\sqrt{\frac{\Delta}{\alpha}}\bs{w},
    \label{model1}
\end{align}
where $\alpha=M/N$ is the measurement rate, $\bs{\phi}\in\mathbb{R}^{M\times N}$ is a known Gaussian matrix with i.i.d. entries, $(\bs{\phi}_{ij})_{i\leq M, j\leq N}\sim \mathcal{N}(0, 1/M)$, $\bs{w}$ is a noise vector with i.i.d.\ components $\bs{w}_{i}\sim\mathcal{N}(0,1)$ for $i\in[M]$, and $\Delta$ is a known variance. We focus on the high-dimensional setting, where $M, N\rightarrow\infty$ with $\alpha$ fixed.

Suppose that $\bs{\theta}$ 
has a prior distribution $\pi=P_0^{\bigotimes N}$ supported on $[-L_\theta,L_\theta]^{N}$. 
Our goal is to generate a sample from 
the posterior distribution
\begin{align}
    \mu_{\bs{y}}(\bs{\theta})
    &:=
\mathbb{P}(\bs{\theta}|\bs{y})
    =\frac{1}{Z(\bs{y})}
    \exp\left(-\dfrac{\alpha}{2\Delta}\|\bs{\phi}\bs{\theta}-\bs{y}\|_2^2\right)P_0^{\bigotimes N}(\bs{\theta}),
\end{align}
where $Z(\bs{y}):=\displaystyle\int\exp\left(-\dfrac{\alpha}{2\Delta}\|\bs{\phi}\bs{\theta}-\bs{y}\|_2^2\right)P_0^{\bigotimes N}(\bs{\theta})d\bs{\theta}$.

Sampling from posterior distribution and understanding its hardness regimes is a major problem in Bayesian learning. Traditional methods such as MCMC or Langevin dynamics may be slow or require log-convexity of the prior. 
Specifically, MCMC tends to require substantial computational resources when applied to high-dimensional problems and suffers from slow convergence due to multimodality of the posterior distribution \citep{dawn2013convergence}.

Recently, \citet{chung2022diffusion} proposed that diffusion model can also be utilized to sample the posterior distribution. Suppose the data noising process $\bs{\theta}_t, t\in[0, T]$ has the form
\begin{align}\label{eq:forward_SDE}
    d\bs{\theta}=-\frac{\beta(t)}{2}\bs{\theta}dt+\sqrt{\beta(t)}d\bs{W},
\end{align}
where $\bs{W}$ is the standard $N$-dimensional Wiener process and $\beta(t)$ is the noise schedule of the process, typically taken to be monotonically increasing linear function of $t$. The data distribution satisfies $\bs{\theta}_0\sim \mathbb{P}(\bs{\theta}|\bs{y})$ and $\bs{\theta}_T\sim \mathcal{N}(0, \bs{I}_N)$. It is possible to use the reverse SDE of \eqref{eq:forward_SDE} to sample from the posterior:
\begin{align}\label{eq:reverse_SDE}
    d\bs{\theta}
    =[&-\frac{\beta(t)}{2}\bs{\theta}
    -\beta(t)(\nabla_{\bs{\theta}_t}\log p_t(\bs{\theta}_t)+\nabla_{\bs{\theta}_t}\log p_t(\bs{y}|\bs{\theta}_t))
    ]dt +
    \sqrt{\beta(t)}d{\bar{\bs{W}}},
\end{align}
where ${\bar{\bs{W}}}$ is a standard Wiener process when time flows backwards from $T$ to $0$. In the sampling algorithm, \citet{chung2022diffusion} use the neural network $s(\bs{\theta}_t, t)$ trained with denoising score matching to replace the score function term $\nabla_{\bs{\theta}_t}p_t(\bs{\theta}_t)$ and $p(\bs{y}|\mathbb{E}_{\bs{\theta}_0\sim p(\bs{\theta}_0|\bs{\theta}_t)}[\bs{\theta}_0])$ to replace the likelihood term $p_t(\bs{y}|\bs{\theta}_t)$ in \eqref{eq:reverse_SDE}. However, $p(\bs{y}|\mathbb{E}_{\bs{\theta}_0\sim p(\bs{\theta}_0|\bs{\theta}_t)}[\bs{\theta}_0])$ is only an approximation of $p_t(\bs{y}|\bs{\theta}_t)$. 
To get a more accurate estimate, we consider using the Approximate Message Passing(AMP) algorithm \citep{david2009message}, which is known for its low complexity and rapid convergence, often achieving near-optimal solutions in a small number of iterations. 
However, AMP still has some limitations. For example, $\bs{\phi}$ should be semi-random matrix ensemble \citep{dudeja2023universality} or $\bs{\phi}$ should satisfy orthogonal rotational invariance in law \citep{fan2022approximate}.

\citet{barbier2020mutual} claims that under some constraints of the prior distribution, the mean squared error (MSE) of the sample generated by AMP algorithm achieves the minimum mean squared error (MMSE) asymptotically when $\Delta$ is within a certain range and gives a rigorous proof in \citet{barbier2019optimal}.
Since \citet{montanari2023sampling} pointed out that sampling based on time-reversal of diffusion processes and on stochastic localization are in fact 
equivalent, 
we extend the analysis in \citet{barbier2020mutual} by giving a rigorous proof of the optimality of AMP in posterior sampling with an additional stochastic localization observation $(\bs{z_t})_{t\geq 0}$:
\begin{align}
\bs{z}_t=t\bs{\theta^*}+\bs{G}_t,
\end{align}
where $\bs{\theta^*}$ is drawn from the posterior distribution $\mu_{\bs{y}}(\bs{\theta})$, and $\bs{G}_t$ is a standard $N$-dimensional Brownian motion and is independent of $\bs{y}$, $\phi$, $\bs{\theta}$, $\bs{\theta^*}$ and $\bs{w}$. Here, we cannot directly apply the results of \citet{barbier2020mutual} if we add the additional stochastic localization observation $(\bs{z_t})_{t\geq 0}$. We use state evolution to prove that AMP is still almost surely optimal with additional observation $(\bs{z_t})_{t\geq 0}$ which is detailed in Theorem \ref{thm1} in Section \ref{sec:con_guaran}.

Recently, \citet{montanari2023posterior} applied stochastic localization and AMP to reconstruct the signal $\bs{\theta}$ from spiked model $\bs{X}=n^{-1}\beta \bs{\theta}\bs{\theta}^\top+\bs{W}$, 
proved convergence for sufficiently large $\beta$, 
and conjectured that the sampling hardness threshold coincides with the infimum $\beta^{-1}$ in the hard phase of mean estimation. Since both spiked model and linear inverse problem deal with reconstructing a high-dimensional signal from the noisy observations, we adopt the same framework on linear inverse problem. 
However, the computation hardness threshold of the linear inverse model \eqref{model1} is not obvious if we follow the idea of \citet{montanari2023posterior} directly. 
In this paper, we verify that the sampling algorithm performs effectively when $\Delta < \Delta_{\rm AMP}$ under the smoothed KL error metric, 
where $\Delta_{\rm AMP}$ is the infimum variance at which AMP does not achieve asymptotically Bayes optimal estimation,
and has been conjectured to be the computation threshold for mean estimation \citep{barbier2020mutual}.

Let $\mu_{\bs{y}, t}(\bs{\theta})$ be the posterior distribution of $\bs{\theta}$ given the measurements $\bs{y}$ and the additional stochastic localization observations $\bs{z}_t$ at time $t$:
\begin{align}
    \mu_{\bs{y}, t}(\bs{\theta}):=&\mathbb{P}(\bs{\theta}\mid \bs{y}, \bs{z}_t)
    =
    \frac{1}{Z(\bs{y}, t)}\exp\left(-\dfrac{\alpha}{2\Delta}\|\phi\bs{\theta}-\bs{y}\|_2^2-\dfrac{1}{2t}\|\bs{z}_t-t\bs{\theta}\|_2^2\right) P_0^{\bigotimes N}
    (\bs{\theta}),
\end{align}
where 
\begin{align}
Z(\bs{y}, t)=\displaystyle\int&\exp\left(-\dfrac{\alpha}{2\Delta}\|\phi\bs{\theta}-\bs{y}\|_2^2-\dfrac{1}{2t}\|\bs{z}_t-t\bs{\theta}\|_2^2\right)P_0^{\bigotimes N}(\bs{\theta})d\bs{\theta}.
\end{align}

The sequence of posteriors $\mu_{\bs{y}, t}(\bs{\theta})$ satisfies the following localization process properties \citep{chen2022localization}: 

\begin{itemize}
    \item The initialization is the posterior distribution $\mu_{\bs{y},0}(\bs{\theta})=\mu_{\bs{y}}(\bs{\theta})$.
    \item The final is a point mass at $\bs{\theta^*}$ sampled from $\mu_{\bs{y}}(\bs{\theta})$, i.e., $\mu_{\bs{y},\infty}(\bs{\theta})=\delta_{\bs{\theta}^*}$. 
    \item $\mu_{\bs{y}, t}(\bs{\theta})$ is a martingale.
\end{itemize}

The process $(\bs{z}_t)_{t\geq 0}$ is also the unique solution of the following stochastic differential equation (section 7.4 from \citet{liptser1977statistics}):
\begin{align}
    d\bs{z}_t=m'(\bs{z}_t, t)dt+d\bs{B}_t, \bs{z}_0=0,
    \label{SDE1'}
\end{align}
where $\bs{B}_t$ is an $N$-dimensional standard Brownian motion and $m'(\bs{z}, t)$ is the posterior mean defined as:
\begin{align}
    m'(\bs{z},t)=\E[\bs{\theta^*}\mid t\bs{\theta^*}+\bs{G}_t=\bs{z}].
\label{e8}
\end{align}
Let 
\begin{align}
    m(\bs{z},t)=\E[\bs{\theta}\mid \bs{\phi\theta}+\sqrt{\Delta/\alpha}\bs{w}=\bs{y}, t\bs{\theta}+\bs{G}_t=\bs{z}].
\label{e9}
\end{align}
We claim that the process $(\bs{z}_t)_{t\geq 0}$ is the unique solution of the following stochastic differential equation, which is rewritten from \eqref{SDE1'}:
\begin{align}
    d\bs{z}_t=m(\bs{z}_t, t)dt+d\bs{B}_t, \bs{z}_0=0,
    \label{SDE1}
\end{align}
since $m'(\bs{z},t)= m(\bs{z},t)$ (note that \eqref{e8} and \eqref{e9} are equivalent since $\bs{\theta^*}$ follows from the posterior distribution whereas $\bs{\theta}$ follows from the prior distribution).

Based on these facts, we can sample from the posterior distribution of $\bs{\theta}$ by tracking the process $\bs{z}_t$. 

\begin{itemize}
    \item First, we discretize the SDE \eqref{SDE1}.
    \item Then, we apply the AMP algorithm to approximate $m(\bs{z}_t,t)$ in each step.
    \item Finally, we use the approximation $\hat{\bs{z}}_T/T$ as a sample from the posterior distribution.
\end{itemize}

We prove convergence of the distribution generated by this algorithm under the smoothed KL divergence, and under the Wasserstein distance assuming a dimension-free bound on the operator norm of the covariance of the posterior distribution.

\subsection{Main contributions}\label{subsec:main_contribution}
We adopt the same framework as in \citet{montanari2023posterior} and provide a sampling algorithm for \eqref{model1} by approximating the posterior mean with the Bayes-AMP solution, and generating a diffusion process. 
We prove $L_2$ convergence results for the AMP solution to the posterior mean in \eqref{SDE1}, by leveraging recent extension of the state evolution analysis \citep{berthier2020state} to the general denoiser case.
As a consequence, 
we prove the convergence in smoothed KL divergence between the target distribution and the distribution generated by the algorithm.
Compared to similar stochastic localization sampling convergence guarantees in recent works for other statistical models \citep{montanari2023posterior}, 
our analysis pushed the computation threshold for sampling to $\Delta_{\rm AMP}$ defined in \citet{barbier2020mutual}, which is believed to be the computation threshold for mean estimation (See Section~\ref{sec_discussion}) and has also been proved to be the computation threshold for a certain class of first order algorithms \citep{cele2020theestimation}.
This characterization of the threshold implies that sampling succeeds for any fixed $\Delta$, as long as $\alpha$ is sufficiently large (Corollary~\ref{large alpha case}),
which is not obvious from previous results \citep{montanari2023posterior} only establishing the existence of a threshold.
We discuss the connection between overlap concentration and Lipschitz constant of $m(\bs{z}_t, t)$ with respect to $\bs{z}_t$ (See Appendix~\ref{sec_w}) and show that the convergence under the Wasserstein distance can be established
assuming that $m_t(\bs{z}_t):=m(\bs{z}_t, t)$ has a dimension free Lipschitz constant.
A key observation allowing us to determine the threshold is Proposition~\ref{prop_unique},
showing that $\Delta<\Delta_{\rm AMP}$ implies two favorable properties:
1) in the Gaussian interpolation argument for proving MSE optimality at fixed $t$, no phase transition occurs;
2) while changing $t$ in the diffusion process, no phase transition occurs, so that the posterior covariance is expected to have a dimension-free operator norm bound.
In contrast, prior works tackle similar issues by either choosing very small noise $\Delta$ (possibly smaller than $\Delta_{\rm AMP}$) 
\citep{montanari2023posterior},
or choosing $N$-dependent time-discretization, raising further issues that are difficult to justify (see Remark~\ref{rem5}).

\subsection{Related works}
{\bf Posterior mean estimation in linear models.}
Over the past decade, a substantial amount of work focused on developing the approximate message passing (AMP) algorithm to solve inverse problems \citep{rangan2011generalized}\citep{bayati2011dynamics},
which holds the promise of making Bayesian inference for high-dimensional models 
computationally tractable, 
in contrast to traditional MCMC algorithms.
The AMP algorithm is an iterative algorithm, which can be viewed as an (asymptotically equivalent) approximation of the message passing algorithm on a bipartite graph \citep{montanari2012graphical}.
For any given number of iterations and as $M,N\to\infty$, the limiting distributions of the state vectors of AMP can be tracked by the state evolution (SE) \citep{bayati2011dynamics}.
Thus by appropriately choosing the denoiser in AMP, 
the asymptotic error may satisfy the same fixed point equation as that of the minimum mean square error (mmse), which is achieved by the posterior mean (Bayes optimal) estimator. 
However, it is nontrivial to establish that the error achieved by AMP actually converges to mmse, as the fixed point may not be unique \citep{barbier2020mutual}. 
It has been conjectured that the algorithmic hard phase is characterized by the parameter regime where the AMP algorithm converges to the Bayes optimal solution (e.g.\ the discussions in \cite{montanari2023posterior} and \cite{barbier2020mutual}).

Variational inference is another popular class of algorithms for Bayesian inference, 
which leverages efficient local (first-order) optimization techniques.
In naive mean-field variational Bayes (NMF), 
a simple product structure is assumed for the posterior distribution, 
enabling fast computation. 
While the asymptotic correctness of NMF has been established under some assumptions such as $M>N$ and well-separatedness \citep{mukherjee2022variational}, 
it is expected that a higher order correction based on Bethe's approximation (also known as the TAP free energy) holds in more general settings \citep{fan2021tap}\cite{krzakala2014variational}.
Interestingly, the stationary points of the TAP free energy coincide with that of the AMP stationary points, and local stationary points that are close to global optimality ensures good posterior estimation \citep{fan2021tap};
however, it is nontrivial to find computationally efficient algorithms for locating such good local optimizers.

{\bf Posterior sampling.}
In Bayesian inference and ensemble learning, it is often useful to generate samples from the posterior distribution.
Sampling is often considered harder than mean estimation,
since given a good sampler (e.g.\ with error measured in the Wasserstein distance) one can estimate the mean by computing the average of the samples \citep{montanari2023posterior}.
Stochastic localization \citep{el2022sampling} recently emerged as a powerful tool for sampling using stochastic gradient methods,
where the gradient is a certain conditional mean. 
Therefore, existing algorithms for mean estimation, such as AMP, can be utilized, with proper modifications that incorporate an additional observation of the localization process.
Closely related to our work are recent deployments of the stochastic localization method in the Sherrington-Kirkpatrick (SK) Model and the spiked PCA \citep{montanari2023posterior}.
These works established the convergence under the Wasserstein distance,
and typically require establishing a dimension-free Lipschitz constant for the AMP estimator.
These conditions can be challenging to establish, 
and place restrictions on the range of noise variance that convergence of the algorithm can be established.
We discuss more in Section~\ref{sec_discussion}.

\section{METHOD}\label{sec:method}

Consider the data distribution model of \eqref{model1}. Suppose $\bs{\theta}$ is
drawn from $\pi=P_0^{\bigotimes N}$ supported on $[-L_\theta,L_\theta]^{N}$ where $P_0$ is known and independent of $N$.
To construct a diffusion process \eqref{SDE1} for sampling, we need to estimate the posterior mean $m(\bs{z}_t,t)$. To do this, we apply the AMP algorithm in equation(202-203) of \citet{berthier2020state} as follows:
\begin{align}
    \bs{r}^{(k)}&=\bs{y}-\bs{\phi}\hat{\bs{m}}^{(k)}+\bs{r}^{(k-1)}b_t^{(k)},\label{AMP1} \\
    \hat{\bs{m}}^{(k+1)}&=\eta(\bs{\phi}^{\top} \bs{r}^{(k)}+\hat{\bs{m}}^{(k)};\tau_{k,t}^2,\bs{z}_t,t),\ \label{AMP2}
\end{align}
where the initialization is given by $\hat{\bs{m}}^{(0)}=\bs{0}$ and 
\begin{align}
    b_t^{(k)} &= \dfrac{1}{N\alpha}\sum_{i=1}^N[\eta'(\bs{\phi}^T \bs{r}^{(k-1)}+\hat{\bs{m}}^{(k-1)};\tau_{k-1, t}^2,\bs{z}_{t}, t)]_i.\label{corr_term}
\end{align}

The denoiser $\eta(\bs{u};\Sigma^2, \bs{z}_t,t)$ is the MMSE estimator associated with the channels $\bs{u}=\bs{x}+\bs{v}\Sigma$ and $\bs{z}_t=t\bs{x}+\bs{g}_t$ where $\bs{u}, \bs{x}, \bs{v}, \bs{z}_t, \bs{g}_t\in\mathbb{R}^N$, $\bs{v}\sim\mathcal{N}(0, \bs{I}_N)$, and $\bs{g}_t$ is a standard N-dimensional Brownian motion. The $l$-th component of the the $N$-dimensional vector $\eta(\bs{u};\Sigma^2, \bs{z}_t,t)$ is given by:
\begin{align}
\frac{\displaystyle\int \bs{x}_lP_0(\bs{x}_l)\exp[-\frac{(\bs{z}_{t,l}-t\bs{x}_l)^2}{2t}-\frac{(\bs{u}_l-\bs{x}_l)^2}{2\Sigma^2}]d\bs{x}_l }{\displaystyle\int P_0(\bs{x}_l)\exp[-\frac{(\bs{z}_{t,l}-t\bs{x}_l)^2}{2t}-\frac{(\bs{u}_l-\bs{x}_l)^2}{2\Sigma^2}]d\bs{x}_l},
\end{align}
and $[\eta'(\bs{u};\Sigma^2, \bs{z}_t,t)]_i$ denotes the the i-th scalar component
of the gradient of $\eta$ w.r.t. its first argument.

$\tau_{k,t}$ is a sequence of variances defined by the following recursion:
\begin{align}
\tau_{0,t}^2&=\frac{\Delta+v}{\alpha}, \text{ where } v=\mathbb{E}_\pi(\|\bs{\theta}\|_2^2/N), \label{rec1}\\   
\tau_{k+1,t}^2&=\frac{\Delta+{\rm mmse}(\tau_{k,t}^{-2},t)}{\alpha},
\end{align}
where 
${\rm mmse}(\Sigma^{-2},t)\nonumber
:=\E_{\theta,\Tilde{G},\Tilde{W}}[|\theta-E[\theta|\theta+\Sigma\Tilde{G},t\theta+\sqrt{t}\Tilde{W}]|^2]$ in which $\Tilde{G}\sim N(0,1)$ and $\Tilde{W}\sim N(0,1)$ are independent. Moreover, we can write ${\rm mmse}(\Sigma^{-2},t)$ in another way as follows:
\begin{align}
    {\rm mmse}^*(\Sigma^{-2})&=\E_{\theta,\Tilde{G}}[|\theta-E[\theta|\theta+\Sigma\Tilde{G}]|^2] \label{mmse^*},\\
    {\rm mmse}(\Sigma^{-2},t)&={\rm mmse}^*(\Sigma^{-2}+t).
\end{align}

\begin{algorithm}[t]
\caption{Sampling Algorithm} \label{Alg1}
\begin{flushleft} 
\textbf{Input:} 
Data $\bs{y}$, Matrix $\bs{\phi}\in\mathbb{R}^{M\times N}$, parameters $(K, T, \delta)$;
\end{flushleft}
\begin{algorithmic} 
\State 1: Set $\tilde{\bs{z}}_0=\bs{0}_N$;
\State 2: Compute $v$, the prior mean of $\|\bs{\theta}\|^2$ divided by $N$;
\State 3: \textbf{for} $l=0, 1, \cdots, T/\delta-1$ \textbf{do}
\State 4: \indent Draw $\bs{B}_{l+1} \sim N(0, \bs{I}_N)$ independent of everything so far;
\State 5: \indent Let $\hat{m}(\tilde{\bs{z}}_{t_l}, t_l)$ be the output of Bayes AMP algorithm \eqref{AMP1}, \eqref{AMP2} with $K$ iterations;
\State 6: \indent Update $\tilde{\bs{z}}_{t_{l+1}}=\tilde{\bs{z}}_{t_l}+\hat{m}(\tilde{\bs{z}}_{t_l}, t_l)\delta+\sqrt{\delta}\bs{B}_{l+1}$;
\State 7: \textbf{end for}
\State 8: \textbf{return} $\bs{\theta}^{\text{alg}}=\tilde{\bs{z}}_{T}/T$

\end{algorithmic}
\begin{flushleft}
\textbf{Output:} $\bs{\theta}^{\text{alg}}$.
\end{flushleft}
\end{algorithm}

As ${\rm mmse}^*$ has an easier form, we will use it instead of ${\rm mmse}$ for calculation.
Actually, $\tau_{k,t}$ satisfies the following state evolution recursion by Lemma \ref{lemma_se}:
\begin{align}
\tau_{0,t}^2=&\frac{\Delta+v}{\alpha}, \text{ where } v=\mathbb{E}_\pi(\|\bs{\theta}\|_2^2/N),\label{se1}\\   
    \tau_{k+1,t}^2=&\frac{\Delta}{\alpha}+\lim_{N\rightarrow\infty}\dfrac{1}{\alpha N}\E_{\bs{\theta}, \tilde{\bs{Z}}}[\|\eta(\bs{\theta}+\tau_{k, t}\tilde{\bs{Z}};\tau_{k, t}^2,\bs{z}_t, t)-\bs{\theta}\|_2^2]
    \label{se2},
\end{align}
where $\tilde{\bs{Z}}\sim N(0, I_N)$. 

Under suitable conditions for $\Delta$, AMP algorithm outputs an asymptotically equivalent estimate of $m(\bs{z}_t,t)$ estimate as $N\to \infty$. Denote $\hat{m}(\bs{z}_t,t)$ as the approximation of $m(\bs{z}_t,t)$. Let $T=L\delta$ and $t_l=l\delta, l=1, \cdots, L$. Now, we construct the following three processes with the true posterior mean and its approximation given by AMP:
\begin{align}
    d\bs{z}_t&=m(\bs{z}_t, t)dt+d\bs{W}_t\label{process1}, \\
    d\hat{\bs{z}}_t&=\hat{m}(\hat{\bs{z}}_t, t)dt+d\bs{W}_t \label{process2},\\
    \tilde{\bs{z}}_{t_l}-\tilde{\bs{z}}_{t_{l-1}}&=\hat{m}(\tilde{\bs{z}}_{t_{l-1}}, t_{l-1})\delta+\bs{W}_{t_l}-\bs{W}_{t_{l-1}},  \quad\text{for } l=1,2,\cdots, L, \label{process3}
\end{align}
where $\bs{W}_t$ is a standard brownian motion, $\hat{m}$ is the output of AMP algorithm \eqref{AMP1}, \eqref{AMP2} with $K$ iterations. 

The first process is the same process as that defined in \eqref{SDE1}. The second process is a continuous process with drift term as estimated posterior mean $\hat{m}(\hat{\bs{z}}_t, t)$ replacing the true posterior mean $m(\bs{z}_t, t)$ in the first process \eqref{process1}. And the last process is the discrete version of the second process \eqref{process2}, which is the update step in our algorithm.

\section{CONVERGENCE GUARANTEES}\label{sec:con_guaran}
Now we provide a bound of the KL divergence between the distribution generated by our algorithm and by the first process\eqref{process1}. 

\begin{theorem}\label{theorem 7}
Consider the linear observation model in \eqref{model1}, 
with fixed measurement rate $\alpha$ and noise level $\Delta$,
and suppose that 
$\bs{\theta}$ has a product prior compactly supported on $[-L_\theta,L_\theta]^{N}$.
Suppose that $\Delta<\Delta_{\rm AMP}$, 
where $\Delta_{\rm AMP}$ is the threshold for mean estimation by the AMP algorithm  explicitly defined after \eqref{fixed point}. Then for any $\epsilon>0$ and $T>0$, there exist $K\in \mathbb{N}$ and $\delta\in\mathbb{R}^+$ such that 
Algorithm~\ref{Alg1} satisfies
the following with probability $1-o_N(1)$:
    \begin{align}
        \dfrac{1}{N}D_{\text{KL}}(P_{\bs{z}_T/T}\|P_{\bs{\theta}^{\text{alg}}})\leq\epsilon.
    \end{align}
\end{theorem}

\begin{remark}
    $P_{\bs{z}_T/T}$ is the distribution of $\bs{\theta}^*+T^{-1/2}\bs{G}$, $\bs{G}\sim \mathcal{N}(0, I_N)$ rather than the true posterior. When $T$ is large enough, we can consider it as a smoothed version of the target distribution.
    In fact, if the true posterior density is smooth, it is possible to show that $P_{\bs{z}_T/T}$ is close to the true posterior using the Fourier analysis (see e.g.\ \cite[Theorem~E.1]{zhang2024minimax}).
\end{remark}

\begin{remark}
    Here we used Theorems 3.1 and 3.2 of \cite{barbier2020mutual}, which is stated for discrete priors. However, as pointed out in \citet[Remark 3.6]{barbier2020mutual}, the prior could be extended to more generalized distributions. The constraint of the prior is relaxed in \citet{barbier2019optimal}. Therefore, we can consider this more general setting, a prior with bounded support.
\end{remark}
\begin{remark}
    We also mention here that under a regularity assumption that the derivative of ${\rm mmse}^*(\Sigma^{-2})$ is continuous at $\Sigma^2=0$, Theorem ~\ref{theorem 7} remains valid for any fixed $\Delta$ , when $\alpha$ is sufficiently large. We state this as a corollary and will prove it in the appendix.
\end{remark}
\begin{corollary}\label{large alpha case}
If the derivative of ${\rm mmse}^*(\Sigma^{-2})$ is continuous at $\Sigma^2=0$, the the sampling algorithm succeeds for any fixed $\Delta$ when $\alpha$ is sufficiently large.
\end{corollary}

{\bf Sketch of proof for Theorem~\ref{theorem 7}}

 {\bf Step 1} For any $\Delta$ and $t$, we can show that AMP convergence implies that MSE equals the solution of a certain fixed point equation.
Denote the asymptotic MSE obtained by AMP at time $t$ and iteration $k$ as
\begin{align}\label{etk}
E_{t,\Delta}^{(k)}:=\lim_{N\rightarrow\infty}\dfrac{1}{N}\|\bs{\theta}-\hat{m}^{(k)}(\bs{z}_t,t)\|_2^2
\end{align}
where $\hat{m}^{(k)}(\bs{z}_t,t)$ is the output of AMP algorithm at time $t$ and iteration $k$. This asymptotic MSE can actually be tracked by the state evolution recursion by Lemma \ref{lemma_einf} given by:
\begin{align}
    E_{t,\Delta}^{(k+1)}={\rm mmse}^*\left(\dfrac{\alpha}{\Delta+E_{t,\Delta}^{(k)}}+t\right).
\end{align}

The function ${\rm mmse}^*$ defined in \eqref{mmse^*} is a monotone function which indicates that $E_{t,\Delta}^{(k)}$ is a decreasing sequence w.r.t. $k$. Therefore, $\lim_{k\rightarrow\infty}E_{t,\Delta}^{(k)}=E_{t,\Delta}^{(\infty)}$ exists and is the fixed point of the equation given as follows:
\begin{align}
    E={\rm mmse}^*\left(\frac{\alpha}{\Delta+E}+t\right).\label{fixed point}
\end{align}
Moreover, following \cite{barbier2020mutual}, we define $\Delta_{\rm AMP}$ as the supremum of all $\Delta$ such that the state evolution fixed point equation \eqref{fixed point} has a unique solution for all noise value in $[0, \Delta]$ when $t=0$. 
With the definition of $\Delta_{\rm AMP}$, it can be shown that $E_{t,\Delta}^{(\infty)}$ is unique for all $t$ if $\Delta\in [0,\Delta_{\rm AMP}]$, which indicates that $\Delta_{\rm AMP}$ is also the threshold for AMP applied in Algorithm \ref{Alg1}:

\begin{proposition}
\label{prop_unique}
For given $(\Delta,t)$,
suppose that  $(\Delta_0,E_0)$ satisfies
\begin{align}
E_0={\rm mmse^*}\left(
\frac{\alpha}{\Delta+E_0}+t
\right)={\rm mmse^*}\left(
\frac{\alpha}{\Delta_0+E_0}
\right).
\end{align}
Then $\Delta_0<\Delta_{\rm AMP}$
is a sufficient condition for $E=E_0$ to be the unique solution to
\begin{align}
E={\rm mmse^*}\left(
\frac{\alpha}{\Delta+E}+t
\right).
\label{e_fixedpt}
\end{align}
In particular, 
$\Delta<\Delta_{\rm AMP}$ is a sufficient condition for $E=E_0$ to be the unique solution.
\end{proposition}
\begin{proof}
Since $\Delta_0<\Delta<\Delta_{\rm AMP}$, $E_0$ is the unique solution to 
\begin{align}
E={\rm mmse^*}\left(
\frac{\alpha}{\Delta_0+E}
\right).
\end{align}
Since $0<{\rm mmse^*}\left(\frac{\alpha}{\Delta_0+0}\right)$ and $\infty>{\rm mmse^*}\left(\frac{\alpha}{\Delta_0+\infty}\right)$, 
using the fact that $E\mapsto {\rm mmse}^*(\frac{\alpha}{\Delta_0+E})$ is increasing,
we see that
\begin{align}
E< {\rm mmse^*}\left(
\frac{\alpha}{\Delta_0+E}
\right),
\quad\forall
E<E_0
\end{align}
and 
\begin{align}
E>{\rm mmse^*}\left(
\frac{\alpha}{\Delta_0+E}
\right),
\quad\forall
E>E_0.
\label{e33}
\end{align}
Now if $E_1>E_0$ is another fixed point for \eqref{e_fixedpt}, 
we have $\frac{\alpha}{\Delta+E_1}+t> \frac{\alpha}{\Delta_0+E_1}$ because 
$\frac{\alpha}{\Delta+E_0}+t= \frac{\alpha}{\Delta_0+E_0}$ and $\Delta>\Delta_0$.
Therefore 
\begin{align}
{\rm mmse^*}\left(\frac{\alpha}{\Delta+E_1}+t\right)\le  {\rm mmse^*}\left(\frac{\alpha}{\Delta_0+E_1}\right),
\end{align}
which, combined with \eqref{e33}, implies that $E_1>{\rm mmse^*}\left(\frac{\alpha}{\Delta+E_1}+t\right)$, a contradiction. 
Similarly, we can also show that it is impossible to have another fixed point $E_1<E_0$.
\end{proof}

{\bf Step 2} Next, we construct the family of observation models in \citet[(37)]{barbier2020mutual}, which is just our model if we let $\tilde{\bs{y}}=\bs{z}_t/t$:
\begin{align}
    \bs{y} &= \bs{\phi\theta}+\sqrt{\frac{\Delta}{\alpha}}\bs{w} \label{model0};\\
    \tilde{\bs{y}}&=\bs{\theta}+\frac{1}{\sqrt{t}}\bs{G}_t.
    \label{model_1}
\end{align}
where $G_t\sim \mathcal{N}(0,I)$ is independent of $(\bs{w,\phi,\theta})$. 
We can then derive the mutual information of the above model, denoted as $i_{t,\Delta}$ as:
\begin{align}
    i_{t,\Delta} = -(\dfrac{\alpha}{2}+1)-\dfrac{1}{N}\E[\ln Z_{t,\Delta}(\bs{y}, \tilde{\bs{y}}) ]
    \label{is}.
\end{align}
where $ Z_{t,\Delta}$ is the partition function defined as follows:
\begin{align}
    &Z_{t,\Delta}(\bs{y}, \tilde{\bs{y}})\nonumber\\
    =&\int d\bs{x}\exp\left(-H_{t,\Delta}(\bs{x}\mid \bs{y}, \tilde{\bs{y}})\right)P_0^{\bigotimes N}(\bs{x}),\\
    &H_{t,\Delta}(\bs{x}\mid \bs{y}, \tilde{\bs{y}})\nonumber\\
    =&\dfrac{\alpha}{2\Delta}\|\phi\bar{\bs{x}}-\sqrt{\frac{\Delta}{\alpha}}\bs{w}\|_2^2+\dfrac{t}{2}\|\bar{\bs{x}}-\frac{1}{\sqrt{t}}\bs{G}_t\|_2^2.
\end{align}

Here we use the notation $\bar{\bs{x}}=\bs{x}-\bs{\theta}$. We can then calculate the derivative of $i_{t,\Delta}$ w.r.t. $t$ by the quantities above. Also, by Lemma 4.5 and Lemma 4.6 of \citet{barbier2020mutual}, the relationship between ${\rm mmse}_{t,\Delta}$ and MI is given, where ${\rm mmse}_{t,\Delta}$ and ${\rm ymmse}_{t,\Delta}$ are given as follows:
\begin{align}
    E_{t,\Delta} := {\rm mmse}_{t,\Delta}&=\dfrac{1}{N}\E[\|\bs{\theta}-\langle\bs{X}\rangle_{t,\Delta}\|_2^2]\\
    {\rm ymmse}_{t,\Delta}&=\dfrac{1}{M}\E[\|\bs\phi(\bs{\theta}-\langle\bs{X}\rangle_{t,\Delta})\|_2^2].
\end{align}

\begin{lemma}\label{lemma4}
Model \eqref{model0} verifies the following I-MMSE relation
\begin{align}
    \dfrac{\partial i_{t,\Delta}}{\partial\Delta^{-1}}=\frac{\alpha}{2}{\rm ymmse}_{t,\Delta}.
\end{align}
\end{lemma}

\begin{lemma}\label{lemma5}
For a.e. $t$,
\begin{align}
    {\rm ymmse}_{t,\Delta} = \frac{E_{t,\Delta}}{1+E_{t,\Delta}/\Delta} + o_N(1).
\end{align}
\end{lemma}

The construction of $i_{t,\Delta}$ as a function of $E_{t,\Delta}$ suggests that $E_{t,\Delta}$ satisfies the following PDE:
    \begin{align}
        \alpha\dfrac{\partial
        }{\partial t}(\dfrac{E}{1+E/\Delta})=\dfrac{\partial E}{\partial \Delta^{-1}},
    \end{align}
    which is the property shared by the fixed point of \eqref{fixed point}. Given $E_{0,\Delta}=E_{0,\Delta}^{(\infty)}$, which is shown in \citet{barbier2020mutual}, as the initialization condition for the PDE, we can prove $E_{t,\Delta}=E_{t,\Delta}^{(\infty)}$ combining the above equation.

{\bf Step 3} Note that we have 
\begin{align}
&\dfrac{1}{N}\mathbb{E}[\|\hat{m}^{(k)}(\bs{z}_t,t)-m(\bs{z}_t,t)\|_2^2]\notag\\
=&
\dfrac{1}{N}\mathbb{E}[\|\hat{m}^{(k)}(\bs{z}_t,t)-\bs{\theta}\|_2^2]
-\dfrac{1}{N}\mathbb{E}[\|\bs{\theta}-m(\bs{z}_t,t)\|_2^2].
\label{e20}
\end{align}
Note that $\dfrac{1}{N}\mathbb{E}[\|\bs{\theta}-m(\bs{z}_t,t)\|_2^2]$ converges to $E_{t,\Delta}$ defined in Step 2.
And for most $z$, $\dfrac{1}{N}\mathbb{E}[\|\hat{m}^{(k)}(\bs{z}_t,t)-\bs{\theta}\|_2^2]$ converges to the fixed point equation solution, $E_{t,\Delta}^{(\infty)}$, shown in Step 1, so it converges to the same quantity.
Thus \eqref{e20} converges to 0 as $k\to\infty$. We summarize this result in the following theorem:

\begin{theorem}\label{thm1}
    For $\forall t$, if $\Delta<\Delta_{\rm AMP}$, then AMP is almost surely optimal, namely,
    \begin{align}
        \lim_{k\to \infty}\lim_{N\to \infty}\dfrac{1}{N}\|\hat{m}^{(k)}(\bs{z},t)-m(\bs{z},t)\|_2^2=0
    \end{align}
    where $\hat{m}^{(k)}$ is the output of AMP algorithm at time $t$ after $k$ iterations.
\end{theorem}

{\bf Step 4} Now we show how to bound the KL divergence between the distribution generated by our algorithm and by the first process\eqref{process1}. First, we bound it by using the Girsanov's theorem \citep{oko2023diffusion} as follows:

\begin{align}
    &\dfrac{1}{N}D_{\text{KL}}(P_{\bs{z}_T/T}\|P_{\tilde{\bs{z}}_T/T})
        \lesssim
         \sum_{l=1}^L\E\int_{t_{l-1}}^{t_l}\dfrac{1}{N}\|m(\bs{z}_t, t)-\hat{m}(\bs{z}_{t_l}, t_l)\|_2^2dt. \label{Girsanov_thm1}
\end{align}

It is easy to note that the right hand side of  \eqref{Girsanov_thm1} can be bounded by the summation of three terms:
\begin{align}
    \sum_{l=1}^L\E\int_{t_{l-1}}^{t_l}\dfrac{1}{N}( &\|m(\bs{z}_t, t)-\hat{m}(\bs{z}_t, t)\|_2^2+ \|\hat{m}(\bs{z}_t, t)-\hat{m}(\bs{z}_{t_l}, t)\|_2^2+\|\hat{m}(\bs{z}_{t_l}, t)-\hat{m}(\bs{z}_{t_l}, t_l)\|_2^2 \label{KL_three_terms1}  )dt.
\end{align}

For the first term, it is just the result of Theorem \ref{thm1} given above. Then, we prove that the maps $\bs{z}_t \mapsto\hat{m}(\bs{z}_t, t)$ and $t \mapsto\hat{m}(\bs{z}_t, t)$ are Lipschitz continuous with high probability shown in Lemma \ref{lemma11} and Lemma \ref{lemma10}. The second and third term can then be bounded by using the definition of diffusion process. We have:
\begin{align}
        &\dfrac{1}{N}\E\|\hat{m}(\bs{z}_t, t)-\hat{m}(\bs{z}_{t_l}, t)\|_2^2 \\
        \lesssim
        & \dfrac{1}{N}\E\|\bs{z}_t-\bs{z}_{t_l}\|_2^2 \\
        =
        &\dfrac{1}{N}\E\|\int_{t_{l-1}}^tm(\bs{z}_s, s)ds+W_{t}-W_{t_l}\|_2^2 .
    \end{align}
    Therefore, we can use the bound of support $L_\theta$ and the property of Brownian Motion to bound the second term. For the last term, using the Lipschitz property of $\hat{m}$ with respect to t, we can bound it in the same way as follows:
    \begin{align}
        &\sum_{l=1}^L\int_{t_{l-1}}^{t_l}\dfrac{1}{N}\|\hat{m}(\bs{z}_{t_l}, t)-\hat{m}(\bs{z}_{t_l}, t_l)\|_2^2dt \\
        \lesssim &
         \sum_{l=1}^L\int_{t_{l-1}}^{t_l} (t-t_{l-1})^2dt 
        \leq
        \dfrac{T\delta^2}{3}.
    \end{align}

    Therefore, for fixed $T$, we can choose $\delta$ small enough such that the summation of these two terms is less than $\epsilon/2$.

However, if we consider other metrics such as the Wasserstein distance, the results may not hold under the same assumptions as Theorem~\ref{theorem 7}. Under a stronger assumption of the Lipschitz continuity of $m(\cdot,t)$, we can prove that the convergence of $W_2$ distance still holds.

\begin{theorem}\label{thm:wass_dist}
    Assume that $\bs{\theta}$ is compactly supported on $[-L_\theta,L_\theta]^{N}$, $\Delta<\Delta_{\rm AMP}$ and $m(z, t)$ is $L_m$-Lipschitz w.r.t. $z$ where $L_m$ is bounded uniformly in $N$. Then for any fixed $T$, there exist $K\in \mathbb{N}$ and $\delta\in\mathbb{R}^+$ such that the following holds: 
    \begin{align}
        \lim_{\delta\rightarrow 0}\lim_{K\rightarrow\infty}\lim_{N\rightarrow\infty}\dfrac{1}{\sqrt{N}}W_2(P_{z_T/T},P_{\tilde{z}_T/T}) =0.
    \end{align}
\end{theorem}
\begin{remark}
    Here, we established the proof assuming that $m(z,t)$ has a dimension-free Lipschitz constant. The validity of this assumption 
    is supported by recent breakthrough results of \cite{brennecke2023operator} and \cite{el2024bounds}
    about the operator norm of covariance matrices for replica symmetric systems; see Appendix~\ref{sec_w}. 
\end{remark}

\begin{remark}
\label{rem5}
In Theorem~\ref{thm:wass_dist}, 
we choose the time precision $\delta$ to be a constant independent of $N$.
In contrast,
\citet{davide2024sampling} mentioned that the Lipschitz constant $L_m$ could be $\Theta(N)$ when phase transition occurs, 
and suggested choosing $\delta=O(1/N)$.
However, a closer inspection reveals that the accumulated error scales as $L_m\delta k$ where $k$ is the number of steps (see for example
\citet[Lemma 4]{davide2024sampling}),
so choosing small $\delta$ won't compensate for diverging $L_m$, 
as $k$ also increases.
After consulting the authors of \citet{davide2024sampling} about this issue,
they suggested that small $\delta$ may only be needed during a phase transition, but since a proof that the phase transition only occurs in an $O(1/N)$ time interval is lacking, 
it remains an interesting challenge to provide a rigorous proof.
By contrast, our Proposition~\ref{prop_unique} shows that phase transition does not occur while changing $t$ as along as $\Delta<\Delta_{\rm AMP}$, thus avoiding choosing an $N$-dependent $\delta$.
\end{remark}

\section{NUMERICAL EXPERIMENTS}\label{sec:numer_exper}

In our experiments, we first consider the following discrete prior distribution:
\begin{align}
    P_0(\theta)=\frac{1}{2}\delta_{-1}+\frac{1}{2}\delta_{1}.
    \label{dist}
\end{align}

\begin{figure}[t]
    \includegraphics[width=\textwidth]{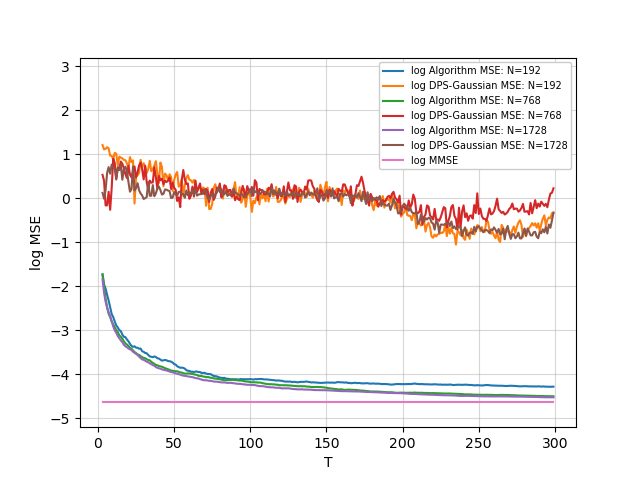}
    \caption{log Algorithm MSE and log DPS-Gaussian MSE for different values of $N$. For this experiment, we let the prior be standard normal distribution and set $\Delta=0.01, \alpha=2, K=50, T=300, \delta=0.1$.}
    \label{fig:alg_mse_dps_mse}
\end{figure}

For the first experiment, we set $M=1000, N=1250, K=20, L=2000, \delta=0.1, v=1$. In Figure \ref{fig:trajectory} in Appendix \ref{appen:results_of_numeric}, we plot the trajectories of the first and the second coordinates of the samples generated by Algorithm \ref{Alg1}: $\tilde{\bs{z}}_{t_{l}}/t_{l}, l\in\{1, \cdots, L\}$ with different $\Delta$. We choose $\Delta=0.01, 1, 10$. For each $\Delta$, we run four independent simulations. In each graph, the circle marker represents the initial value of the sample, and the star marker represents the final value of the sample. Notice that $(\tilde{\bs{z}}_{t_{l}}/t_{l})_1$ and $(\tilde{\bs{z}}_{t_{l}}/t_{l})_2$ converge to $+1$ or $-1$ as $l\rightarrow\infty$.

Then we consider the following continuous prior distribution which is a standard normal distribution:
\begin{align}
    P_0(\theta) = \frac{1}{\sqrt{2\pi}}\exp(-\frac{\theta^2}{2}).
\end{align}

For the second experiment, we fix $\alpha=2$ and set $K=50, T=300, \delta=0.1, v=1$. For $\Delta=0.01$, we plot the log MSE of the samples generated by our algorithm with different values of $N$ and $M$. We choose $N=192, 768, 1728$, which implies that $M=384, 1536, 3456$ since $\alpha=2$. We perform comparison with DPS-Gaussian method \citep{chung2022diffusion}.

In Figure \ref{fig:alg_mse_dps_mse}, we observe a gap between the MSE of the samples generated by DPS-Gaussian and MMSE since \citet{chung2022diffusion} used $p(\bs{y}|\mathbb{E}_{\bs{\theta}_0\sim p(\bs{\theta}_0|\bs{\theta}_t)}[\bs{\theta}_0])$ to approximate the likelihood $p_t(\bs{y}|\bs{\theta}_t)$. In contrast, for our Algorithm, we find that within each fixed $N$, the algorithm MSE will gradually approach the MMSE as $T$ increases.

\section{DISCUSSION}\label{sec_discussion}
Closely related to our work are recent deployments of the stochastic localization method in the Sherrington-Kirkpatrick (SK) Model and the spiked PCA \citep{montanari2023posterior}.
These works established the convergence under the Wasserstein distance,
and typically require establishing a dimension-free Lipschitz constant for the AMP estimator.
These conditions can be challenging to establish, 
and place restrictions on the range of noise variance that convergence of the algorithm can be established.
The challenges include:
1) Ensure that $m(\cdot, t)$ is Lipschitz continuous;
see the choice of $\beta_0$ in Section~G of \cite{montanari2023posterior}.
2) The noise variance also needs to be small enough such that a fixed point equation involving the localization time $t$ has a unique solution for all $t>0$ \cite[Proposition~D.1]{montanari2023posterior}.
This condition may appear stronger than assuming the uniqueness of the fixed point equation in the mean estimation problem, which is just the special case of $t=0$. It is conjectured that the threshold for efficient sampling is the infimum variance at which AMP does not achieve asymptotically Bayes optimal estimation \cite[Remark~3.2]{montanari2023posterior}.
This conjecture is natural since the additional stochastic localization process effectively reduces the noise variance.
In this paper we partially resolve this conjecture (for the setting of linear models), 
by adopting slightly different proof strategies and error metrics that bypass or resolve the above mentioned challenges in establishing Lipschitz continuity and uniqueness of the fixed point.
For challenge 1), we leverage a monotone convergence property of the mmse in state evolution, which bypasses the more sophisticated AMP contraction property,
and we use the Girsanov theorem to prove convergence in the smoothed KL divergence, as opposed to the Wasserstein distance.
For 2), we make an observation in Proposition~\ref{prop_unique} that the uniqueness of the fixed point equation involving the stochastic localization process can be guaranteed under the uniqueness of the fixed point equation without the stochastic localization process.

For comparison, in Appendix~\ref{sec_w} we provide a proof of convergence under the Wasserstein distance, following a similar style of arguments in \cite{montanari2023posterior}, under an assumption of the Lipschitz continuity of $m(\cdot,t)$. 
In \cite{montanari2023posterior}, a similar condition for the Lipschitz continuity of $\hat{m}^{(k)}(\cdot,t)$ for large $k$ was adopted, which relies on sophisticated results on contractive properties of the AMP iteration from prior works,
which is one of the reasons why \cite{montanari2023posterior} requires sufficiently small noise variance.
In Appendix~\ref{sec_overlap} we note a connection between Lipschitz continuity of $m(\cdot,t)$ and the well-studied problem of concentration of the \emph{overlap}, which suggests that such Lipschitz conditions are likely to be true in the ``Bayes optimal'' setting,
but may not hold in more general settings.
We remark that the Overlap Gap Property \citep{gamarnik2021overlap} has recently been proposed as a theory for computation hardness.

\bibliographystyle{plainnat}
\bibliography{yourbibfile.bib}

\newpage
\appendix

\section*{\centering\LARGE\textbf{Supplementary Materials}}
 
\section{PROOFS}
\subsection{Proof of Theorem \ref{thm1}}\label{appen:proof_of_thm2}

To prove this theorem, we need to calculate the state evolution recursion to show how AMP converges. The state evolution here has a simple intuitive explanation. Suppose we replace $\phi$ by i.i.d. sampled random matrix $\phi^k$ for each iteration. $\bs{y}^k$ is replaced by $\phi^k\bs{m_0}+\sqrt{\Delta/\alpha}\bs{w}$. Then, the iterations become:

\begin{align}
        \bs{r}^{(k)}&=\bs{y}^{k}-\phi^k\hat{m}^{(k)}\, \\
        \hat{m}^{(k+1)}&=\eta_k(\phi^{k\top} r^{(k)}+\hat{m}^{(k)}). 
\end{align}
By eliminating $r^{(k)}$, the recursion can be given by:

\begin{align}
    \hat{m}^{(k+1)}&=\eta_k(\phi^{k\top} \bs{y}^{k}+(I-\phi^{k\top}\phi^{k})\hat{m}^{(k)})\\
    &=\eta_k(m_0+\sqrt{\Delta/\alpha}\phi^{k\top}\bs{w}+(I-\phi^{k\top}\phi^{k})(\hat{m}^{(k)}-m_0)).\label{se20}
\end{align}
Therefore, if we denote $\hat{\tau}_k=\lim_{N\to\infty}\frac{1}{N}\|\hat{\bs{m}}^{(k)}-\bs{m}_0\|_2^2$, we can use $\hat{\tau}_k$ to describe the convergence of each term in \eqref{se20}. We have $(I-\phi^{k\top}\phi^{k})(\hat{\bs{m}}^{(k)}-\bs{m}_0))$ converges to a vector with i.i.d. entry with mean 0 and variance $\hat{\tau}_k/\alpha$. Besides, by law of large numbers $\phi^{k\top}\bs{w}$ converges to $1/\alpha$. Thus, we have
\begin{align}
    \tau_k &= \dfrac{\hat{\tau}_k}{\alpha}+\frac{\Delta}{\alpha}\\
    \lim_{N\to\infty}\hat{m}^{(k+1)} &= \eta_k(m_0+\tau_k\bs{Z})\\
    \hat{\tau}_{k+1}&=\E(\eta_k(m_0+\tau_k\bs{Z})-m_0)^2
\end{align}
which is just the state evolution given in \eqref{se1} and \eqref{se2}.
However, the correlations between $\phi$ and $\hat{\bs{m}}^{(k)}$ can not be neglected. The Onsager terms are then used to asymptotically eliminate these correlations shown as $b_t^{(k)}$. Remark 7.2 in \citet{berthier2020state} shows that it can be given by:
\begin{align}
    b_t^{(k)} = \dfrac{1}{N\alpha}\sum_{i=1}^N[\eta'(\bs{\phi}^T \bs{r}^{(k-1)}+\hat{\bs{m}}^{(k-1)};\tau_{k-1, t}^2,\bs{z}_{t}, t)]_i.\label{b_t_form}
\end{align}
Now, we will show that the state evolution here coincides with $\tau_{k,t}$ defined in \eqref{rec1}.
\begin{lemma}\label{lemma_se}
    $\tau_{k,t}$ satisfies the state evolution:
    \begin{align}
    \tau_{0,t}^2&=\frac{\Delta}{\alpha}+\frac{v}{\alpha}, \text{ where } v=\mathbb{E}_\pi(\|\bs{\theta}\|_2^2/N),\label{lemma_se1}\\   
    \tau_{k+1,t}^2&=\frac{\Delta}{\alpha}+\lim_{N\rightarrow\infty}\dfrac{1}{\alpha N}\E_{\bs{\theta}, \tilde{\bs{Z}}}[\|\eta(\bs{\theta}+\tau_{k, t}\tilde{\bs{Z}};\tau_{k, t}^2,z_t, t)-\bs{\theta}\|_2^2].\label{lemma_se2}
\end{align}
\end{lemma}

\begin{proof}
We apply induction method to prove the result. For $k=0$, we can find that \eqref{lemma_se1} is just the definition of $\tau_{k,0}$ in \eqref{rec1}.

Supposing $\eqref{lemma_se2}$ holds for $k$, let's consider the case the for $k+1$:
    \begin{align}
    &\frac{\Delta}{\alpha}+\lim_{N\rightarrow\infty}\dfrac{1}{\alpha N}\sum_{i=1}^N\E_{\theta_i, \tilde{Z}_i}[|\eta(\theta_i+\tau_{k, t}\tilde{Z}_i;\tau_{k, t}^2,z_{t, i}, t)-\theta_i|^2] \\
    \overset{a.s.}{=}&\frac{\Delta}{\alpha}+\dfrac{1}{\alpha}\E_{z_{t,i}}\E_{\theta_i, \tilde{Z}_i}[|\eta(\theta_i+\tau_{k, t}\tilde{Z}_i;\tau_{k, t}^2,z_{t, i}, t)-\theta_i|^2] \\
    =&\frac{\Delta}{\alpha}+\dfrac{1}{\alpha}\E_{\theta_i,\Tilde{Z}_i,\Tilde{W}}[|\theta_i-E[\theta_i|\theta_i+\tau_{k,t}\Tilde{Z}_i,t\theta_i+\sqrt{t}\Tilde{W}]|^2]\\
    =&\frac{\Delta}{\alpha}+\frac{{\rm mmse}(\tau_{k,t}^{-2},t)}{\alpha}\\
    =&\tau_{k+1,t}^2,
\end{align}

where the first equality holds by the Strong Law of Large Numbers. Thus, the state evolution holds for all $k$. 
\end{proof}

With the state evolution, we can now describe the asymptotic MSE of AMP by $\tau_{k,t}$.

\begin{lemma}\label{lemma_einf}
    Recall $E_{t,\Delta}^{(k)}$ in \eqref{etk}, we have:
    \begin{align}
    E_{t,\Delta}^{(k+1)}={\rm mmse}^*\left(\dfrac{\alpha}{\Delta+E_{t,\Delta}^{(k)}}+t\right).
\end{align}
\end{lemma}
\begin{proof}
    By Lemma \ref{lemma_se} and remark 7.1 in \citet{berthier2020state}, we have
\begin{align}
    E_{t,\Delta}^{(k+1)}&=\lim_{N\rightarrow\infty}\dfrac{1}{N}\|\hat{m}^{(k+1)}-\bs{\theta}\|_2^2\\
    &\overset{p}{\rightarrow}\lim_{N\rightarrow\infty}\dfrac{1}{N}\E_{\bs{\theta}, \tilde{\bs{Z}}}[\|\eta(\bs{\theta}+\tau_{k, t}\tilde{\bs{Z}};\tau_{k, t}^2,\bs{z}_t, t)-\bs{\theta}\|_2^2]\\
    &=\alpha(\tau_{k+1, t}^2-\frac{\Delta}{\alpha})\\
    &={\rm mmse}^*(\tau_{k, t}^{-2}+t)\\
    &={\rm mmse}^*\left(\dfrac{\alpha}{\Delta+E_{t,\Delta}^{(k)}}+t\right).
\end{align}
\end{proof}

With Lemma \ref{lemma_einf} and the fact that ${\rm mmse}^*$ is a monotonically decreasing function, we have $\lim_{k\rightarrow\infty}E_{t,\Delta}^{(k)}=E_{t,\Delta}^{(\infty)}$ exists and is the fixed point of the equation given as follows:
\begin{align}
    E={\rm mmse}^*(\alpha/(\Delta+E)+t).
\end{align}

This $E_{t,\Delta}^{(\infty)}$ is actually the true MMSE under some assumptions. To show this, we would apply Lemma \ref{lemma4} and Lemma \ref{lemma5}. With the property of relationship between MI and MMSE, we can derive the following theorem with Proposition \ref{prop_unique}, which indicates the convergence of AMP algorithm.

\begin{theorem}
For $\Delta<\Delta_{AMP}$ and a.e. $t$, we have
\begin{align}
    E_{t,\Delta} = E_{t,\Delta}^{(\infty)},
\end{align}
where $E_{t,\Delta}$ denotes the limiting normalized mmse for the observation model in \eqref{model0}-\eqref{model_1}, 
and $E_{t,\Delta}^{(\infty)}$ denotes the limiting normalized mse of the AMP algorithm.
Therefore, the $\hat{m}$ generated by the AMP algorithm will converge to $m$ when the number of iterations grow to infinity, i.e.:
\begin{align}
    \lim_{k\to \infty}\lim_{N\to \infty}\frac1{N}\|\hat{m}^{(k)}(z,t)-m(z,t)\|_2^2=0.
\end{align}
\end{theorem}

\begin{proof}
    First, let's consider a special case when $t=0$. That is we only have the first observation \eqref{model0}. This is the case proved in \citet{barbier2020mutual}, and we have
    \begin{align}
        \dfrac{d}{d\Delta^{-1}}\lim_{N\to \infty}i_{0, \Delta} = \dfrac{\alpha}{2}\dfrac{E_{0, \Delta}}{1+E_{0,\Delta}/\Delta} = \dfrac{\alpha}{2}\dfrac{E_{0,\Delta}^{(\infty)}}{1+E_{0,\Delta}^{(\infty)}/\Delta}.
    \end{align}
    For general $t>0$, we can write $i_{t, \Delta}$ as:
    \begin{align}
        i_{t,\Delta} = i_{0,\Delta} +\int_{0}^{t}\dfrac{\partial i_{s,\Delta}}{\partial s}ds.\label{48}
    \end{align}
    Denote the expectation w.r.t. the Gibbs measure as:
\begin{align}
    \langle g(\bs{X})\rangle_{t,\Delta}:=\int d\bs{x}P_t(\bs{\theta}=\bs{x}\mid \bs{y}, \tilde{\bs{y}})g(\bs{x}).
\end{align}
where $P_t$ is the posterior of $\bs{\theta}$ given $\bs{y}$ and $\tilde{\bs{y}}$. By computing the derivative of $i_{t,\Delta}$ defined as \eqref{is}, we can find that:
    \begin{align}
        \dfrac{\partial i_{t,\Delta}}{\partial t}=\dfrac{1}{2N}\sum_{i=1}^{N}\E[\langle\bar{X}_i^2-\dfrac{1}{\sqrt{t}}\bar{X}_iG_i\rangle_{t,\Delta}].\label{dit}
    \end{align}
    Applying integration by parts w.r.t. $G_i$, we can simplify \eqref{dit} as:
    \begin{align}
        \lim_{N\to\infty}\dfrac{\partial i_{t,\Delta}}{\partial t}
        =\lim_{N\to\infty}\dfrac{1}{2N}\sum_{i=1}^{N}\E[\langle\bar{X}_i\rangle_{t,\Delta}^2]=\dfrac{1}{2}E_{t,\Delta}.
        \label{e65}
    \end{align}
    Taking the derivative of \eqref{48} with respect to $\Delta^{-1}$ on both sides when $N\to \infty$, by Lemma \ref{lemma4} and Lemma \ref{lemma5}, we have
    \begin{align}
        \dfrac{\alpha}{2}\dfrac{E_{t,\Delta}}{1+E_{t,\Delta}/\Delta} = \dfrac{\alpha}{2}\dfrac{E_{0,\Delta}}{1+E_{0,\Delta}/\Delta} +\dfrac{1}{2}\int_{0}^{t}\dfrac{\partial E_{s,\Delta}}{\partial\Delta^{-1}}ds.
        \label{e66}
    \end{align}
    The proof proceeds by showing that $E_{t,\Delta}^{(\infty)}$ and $E_{t,\Delta}$ satisfy the same PDE in Lemma~\ref{lem8}.
    First, we observe that if $E:=E(t,\Delta)$ is any fixed point of \eqref{fixed point}, then
    \begin{align}
        \dfrac{\alpha}{2}\dfrac{\partial}{\partial t}(\dfrac{E}{1+E/\Delta})=\dfrac{1}{2}\dfrac{\partial E}{\partial\Delta^{-1}}. \label{eq:fixed_E}
    \end{align}
    Indeed, denote
    \begin{align}
        A :=  ({\rm mmse}^*)'(\dfrac{\alpha}{\Delta+E}+t).
    \end{align}
    Since $E=E(t,\Delta)$ satisfies
    \begin{align}
        E={\rm mmse}^*(\dfrac{\alpha}{\Delta+E}+t),
    \end{align}
    by the chain rule, we have
    \begin{align}
        \dfrac{\partial E}{\partial \Delta^{-1}}&=A\alpha\dfrac{\Delta^2-\dfrac{\partial E}{\partial \Delta^{-1}}}{(\Delta+E)^2};\\
        \dfrac{\partial E}{\partial t}&=A\left(\dfrac{-\alpha \dfrac{\partial E}{\partial t}}{(\Delta+E)^2}+1\right),
    \end{align}
    which implies
    \begin{align}
        \dfrac{\partial E}{\partial \Delta^{-1}}&=\dfrac{A\alpha\Delta^2}{(\Delta+E)^2+\alpha A};\\
        \dfrac{\partial E}{\partial t}&=\dfrac{A(\Delta+E)^2}{(\Delta+E)^2+\alpha A}.
    \end{align}
    Plugging them into both sides of \eqref{eq:fixed_E}, we obtain
    \begin{align}
        {\rm L.H.S.}&=\dfrac{\alpha}{2}\dfrac{\partial E/\partial t}{(1+E/\Delta)^2}=\dfrac{\alpha}{2}\dfrac{1}{(1+E/\Delta)^2}\dfrac{A(\Delta+E)^2}{(\Delta+E)^2+\alpha A}=\dfrac{\alpha}{2}\dfrac{A\Delta^2}{(\Delta+E)^2+\alpha A};
        \\
        {\rm R.H.S.}&=\dfrac{1}{2}\dfrac{A\alpha\Delta^2}{(\Delta+E)^2+\alpha A},
    \end{align}
    which verifies \eqref{eq:fixed_E}.
    This suggests that $\dfrac{E_{t,\Delta}^{(\infty)}}{1+E_{t,\Delta}^{(\infty)}/\Delta}$ follows the same track as $\dfrac{E_{t, \Delta}}{1+E_{t, \Delta}/\Delta}$, i.e.,
    \begin{align}
         \dfrac{\alpha}{2}(\dfrac{E_{t,\Delta}^{(\infty)}}{1+E_{t,\Delta}^{(\infty)}/\Delta}- \dfrac{E_{0,\Delta}^{(\infty)}}{1+E_{0,\Delta}^{(\infty)}/\Delta})
         =\dfrac{\alpha}{2}\int_{0}^{t}\dfrac{\partial}{\partial s}(\dfrac{E_{s,\Delta}^{(\infty)}}{1+E_{s,\Delta}^{(\infty)}/\Delta})ds
         =\dfrac{1}{2}\int_{0}^{t}\dfrac{\partial E_{s,\Delta}^{(\infty)}}{\partial \Delta^{-1}}ds.
         \label{e76}
    \end{align}
    As $E_{t,\Delta}^{(\infty)}$ is the unique fixed point of \eqref{fixed point} and initial condition $E_{0, \Delta} = E_{0,\Delta}^{(\infty)}$ holds for any $\Delta<\Delta_{\rm AMP}$, we have $E_{t, \Delta} = E_{t,\Delta}^{(\infty)}$ as a result of the common property in PDE.
\end{proof}

The property is given as follows:

\begin{lemma}\label{lem8}
    The solution of the following PDE is unique:
    \begin{equation}
        \left\{\begin{array}{c}
            \alpha\dfrac{\partial}{\partial t}(\dfrac{E}{1+E/\Delta})=\dfrac{\partial E}{\partial \Delta^{-1}},\\
            E(\delta, 0)= E_{0,\Delta}.
    \end{array}\right. \label{PDE1}
    \end{equation}
\end{lemma}

\begin{proof}
    For any solution $E$ of the PDE \eqref{PDE1}, we have 
    \begin{align}
        \alpha \dfrac{\partial E}{\partial t}\dfrac{1}{(\Delta+E)^2}=-\dfrac{\partial E}{\partial\Delta}.
    \end{align}
    $\partial E/\partial t=0$ indicates that $\partial E/\partial\Delta=0$, then $E$ is a constant function which is unique. Now we suppose  $\partial E/\partial t\neq0$.
    
    Suppose $\Delta$ is a function of $t$ such that $E(\Delta(t), t)=E(\Delta(0), 0)$ for all $t$. 
    Then
    \begin{align}
        0=&\dfrac{\partial E(\Delta(t), t)}{\partial t}\\
        =
        &\dfrac{\partial E}{\partial\Delta}\dfrac{\partial \Delta}{\partial t}+\dfrac{\partial E}{\partial t}\\
        =
        & \dfrac{\partial E}{\partial t}(1-\dfrac{\alpha}{(\Delta+E)^2}\dfrac{\partial\Delta}{\partial t})
    \end{align}
    implies 
    \begin{align}
        \dfrac{\alpha}{(\Delta+E)^2}\dfrac{\partial \Delta}{\partial t}-1=0.
    \end{align}
    This indicates that $\alpha/(\Delta+E)+t$ is a constant as a function of $t$ since
    \begin{align}
        &\dfrac{\partial}{\partial t}(\dfrac{\alpha}{\Delta(t)+E(\Delta(t), t)}+t) \\
        =&-\dfrac{\alpha}{(\Delta+E)^2}\dfrac{\partial \Delta}{\partial t}+1\\
        =& 0.
    \end{align}
    Therefore
    \begin{align}
        \dfrac{\alpha}{\Delta(0)+E(\Delta(0), 0)}
        =\dfrac{\alpha}{\Delta(t)+E(\Delta(t), t)}+t.
    \end{align}
    If we let 
    \begin{align}
        \Delta(t)= \dfrac{\alpha}{\dfrac{\alpha}{\Delta(0)+E(\Delta(0), 0)}-t}-E(\Delta(0), 0)\label{Delta_t},
    \end{align}
    then for any solution $E$ of PDE \eqref{PDE1}, there always exists $\Delta(t)$ such that
    \begin{align}
        E(\Delta(t), t)=E(\Delta(0), 0) \text{ \ \ \ for all \   } t,
    \end{align}
    where $\Delta(0)$ and $E(\Delta(0), 0)$ are given previously. 
    
    Suppose the PDE \eqref{PDE1} has two solutions $E_1$ and $E_2$.
    Then there exist $\Delta_1$ and $\Delta_2$ such that
    \begin{align}
        E_1(\Delta_1(t), t)=E_1(\Delta_1(0), 0) \\
        E_2(\Delta_2(t), t)=E_2(\Delta_2(0), 0).
    \end{align}
    If we let $\Delta_1(0)=\Delta_2(0)=\delta$, then $E_1(\Delta_1(0), 0)=E_1(\delta, 0)=E_{0,\Delta}=E_2(\delta, 0)=E_2(\Delta_2(0), 0)$ which implies $\Delta_1(t)=\Delta_2(t)$. Since any $(\Delta, t)$ can be tracked by \eqref{Delta_t} , we have $E_1\equiv E_2$.
\end{proof}

\subsection{Proof of Theorem \ref{theorem 7}}\label{appen:proof_of_thm1}

In order to prove Theorem \ref{theorem 7}, we introduce the following lemmas.

\begin{lemma}\label{lemma_W_properties}
 With probability at least $1-o_N(1)$, the following inequality holds for some $C_1(\epsilon,T),C_2(\epsilon,T)$:
\begin{align}
 \frac{1}{\sqrt{N}}\sup_{t\in [0,T]}\|W_t\|_2 & \le C_1,
 \label{BM_property}\\
\frac{1}{N}\int_{t_{l-1}}^{t_l}\|W_t-W_{t_{l-1}}\|_2^2dt & \le C_2\delta^{2} , \  for \ l=1,2,...,L.
\label{BM_prop2}
\end{align}
\end{lemma}

\begin{proof}
Actually, as $W_t$ is a $N$ dimensional vector and all the dimensions are independent from each other, we can write $\|W_t\|_2^2$ as the sum of $W_{t,i}^2,i=1,2,...,M$, which are $M$ i.i.d. squared normal distributions. We have:
\begin{align}
\frac{1}{N}\sup_{t\in [0,T]}\|W_t\|_2^2&=\sup_{t\in [0,T]}\frac{1}{N}\sum_{i=1}^MW_{t,i}^2\\
&\le \frac{M}{N}\frac{1}{M}\sum_{i=1}^M\sup_{t\in [0,T]}W_{t,i}^2
\end{align}

By S.L.L.N, $\frac{1}{M} \sum_{i=1}^M \sup_{t\in [0,T]}W_{t,i}^2\to T^2\E \{ \sup_{t\in [0,T]} W_{t,1}^2 \}$ almost surely. And as $M/N \to \alpha$, we can always find $C_{\alpha}$ to bound $M/N$. Thus, we only need to let 
\begin{align}
    C_1=C_{\alpha} T^2\E\{\sup_{t\in [0,T]}W_{t,1}^2\}
\end{align}

For the second constant $C_2$, if we write $\|W_t-W_{t_{l-1}}\|_2^2=\sum_{i=1}^M(W_{t,i}-W_{t_{l-1},i})^2$ and use the fact that $\int_{t_{l-1}}^{t_l}\|W_t-W_{t_{l-1}}\|_2^2dt\le \delta \sup_{t\in [t_{l-1},t_{l}]}\|W_t-W_{t_{l-1}}\|_2^2$. We can get $C_2$ in the same way as $C_1$ by S.L.L.N.
\end{proof}

\begin{lemma}\label{lemma_phi_op}
    With probability $1-o_N(1)$, we have 
    \begin{align}\label{op_condition}
        \|\bs{\phi}\|_{\text{op}}\leq C(1+\dfrac{2}{\sqrt{\alpha}}):=C_{\phi}
    \end{align}
    where $C$ is a constant.
\end{lemma}

\begin{proof}
    Applying Theorem 4.4.5 in \citet{Vershynin_2018}.
\end{proof}

\begin{lemma}\label{lemma11}
    When condition \eqref{BM_property} satisfies, there exists a constant $C_3$ only depends on $C_1$ such that:
    \begin{align}
        \dfrac{1}{\sqrt{N}}\|\hat{m}(\hat{z_t}, t)-\hat{m}(\hat{z}_t, t_{l-1})\|_2\leq C_3|t-t_{l-1}|, \ for \ t\in[t_{l-1}, t_l), \ l=1,2,...,L.
        \label{mhat_part2}
    \end{align}
\end{lemma}
\begin{proof}
    For $t\in[t_{l-1}, t_l)$, $\|\hat{z}_t\|_2=\|\displaystyle\int_0^t\hat{m}(\hat{z}_t, t)dt+W_t\|_2\leq\|\displaystyle\int_0^t\hat{m}(\hat{z}_t, t)dt\|_2+\|W_t\|_2\leq\sqrt{N}L_\theta T+\sqrt{N}C_1=\sqrt{N}(L_\theta T+C_1)$.

    Consider the $i$-th component of $\hat{m}$,
    Let
    \begin{align}
        \langle A(\theta)\rangle_z=\dfrac{\displaystyle\int A(\theta)\exp(-\dfrac{(\hat{z}_{t,i}-\theta)^2(\tau_{k,t}^{-2}+t)}{2})\pi(\theta
        )d\theta}{\displaystyle\int \exp(-\dfrac{(\hat{z}_{t,i}-\theta)^2(\tau_{k,t}^{-2}+t)}{2})\pi(\theta
        )d\theta}
    \end{align}
    Since $\hat{m}(\hat{z}_{t,i},t)=E[\theta|\theta+(\tau_{k,t}^{-2}+t)^{-1/2}G=\hat{z}_{t,i}]$, then $|\dfrac{d\hat{m}(\hat{z}_{t,i},t)}{dt}|=|-\dfrac{1}{2}(\langle\theta^3\rangle_z-\langle\theta\rangle_z\langle\theta^2\rangle_z-2\hat{z}_{t,i}(\langle\theta^2\rangle_z-\langle\theta\rangle_z^2))|\leq\dfrac{1}{2}(2L_\theta^3+2|\hat{z}_{t,i}|L_\theta^2)=L_\theta^3+|\hat{z}_{t,i}|L_\theta^2$ which means $|\hat{m}(\hat{z}_{t,i},t)-\hat{m}(\hat{z}_{t,i},t_{l-1})|\leq(L_\theta^3+|\hat{z}_{t,i}|L_\theta^2)|t-t_{l-1}|\leq\sqrt{(L_\theta^3+|\hat{z}_{t,i}|^2)(L_\theta^3+L_\theta^4)}|t-t_{l-1}|$. Therefore,
    \begin{align}
        &\dfrac{1}{\sqrt{N}}\|\hat{m}(\hat{z_t}, t)-\hat{m}(\hat{z}_t, t_{l-1})\|_2 \\
        =&\dfrac{1}{\sqrt{N}}\sqrt{\displaystyle\sum_{i=1}^N|\hat{m}(\hat{z}_{t,i},t)-\hat{m}(\hat{z}_{t,i},t_{l-1})|^2}\\
        \leq&\dfrac{1}{\sqrt{N}}\sqrt{\displaystyle\sum_{i=1}^N(L_\theta^3+|\hat{z}_{t,i}|^2)(L_\theta^3+L_\theta^4)(t-t_{l-1})^2} \\
        =&\dfrac{1}{\sqrt{N}}\sqrt{(L_\theta^3+L_\theta^4)\displaystyle\sum_{i=1}^N(L_\theta^3+|\hat{z}_{t,i}|^2)}|t-t_{l-1}|\\
        \leq&\dfrac{1}{\sqrt{N}}\sqrt{(L_\theta^3+L_\theta^4)(NL_\theta^3+\|\hat{z}_t\|_2^2)}|t-t_{l-1}|\\
        \leq & \sqrt{(L_\theta^3+L_\theta^4)(L_\theta^3+(L_\theta T+C_1)^2)}|t-t_{l-1}|.
    \end{align}
    Finally, to complete the proof, we only need to let $C_3=\sqrt{(L_\theta^3+L_\theta^4)(L_\theta^3+(L_\theta T+C_1)^2)}$.
\end{proof}

\begin{lemma}
    Suppose $f(x)$ is $L_f$-Lipschitz, and $g(x)$ is $L_g$-Lipschitz. Moreover, if  $\|f\|_2\leq M_f$ and $\|g\|_2\leq M_g$, then 
    \begin{align}
        f(x)g(x) \text{ is } L_{fg}\text{-Lipschitz}
    \end{align}
    where $L_{fg}=L_fM_g+L_gM_f$.
\end{lemma}

\begin{lemma}\label{lemma10}
   When condition \eqref{op_condition} satisfies, there exists a constant $C_4(K, T)$ which only depends on $K, T$ such that:
   \begin{align}
       \|\hat{m}(\bs{z}_1,t)-\hat{m}(\bs{z}_2, t)\|_2\leq C_4\|\bs{z}_1-\bs{z}_2\|_2.
   \label{mhat_part1}
   \end{align}
\end{lemma}

\begin{proof}
    First, for $1\leq i\leq N$, notice that $\eta(\bs{y}^*_i; \tau_{k,t}^2):=E[\theta\mid \theta+(\tau_{k,t}^{-2}+t)^{-1/2}G=\bs{y}^*_i]$, where $G\sim N(0,1)$, we have $\eta'(\bs{y}^*_i;\tau_{k,t}^2)=d\eta(\bs{y}^*_i; \tau_{k,t}^2)/d\bs{y}^*_i=(\tau_{k,t}^{-2}+t)\text{Var}(\theta\mid \theta+(\tau_{k,t}^{-2}+t)^{-1/2}G=\bs{y}^*_i)\leq(\alpha/\Delta+T)L_\theta^2$. Therefore, $\eta(\bs{y}^*_i;\tau_{k,t}^2)$ is $(\alpha/\Delta+T)L_\theta^2$-Lipschitz, and consequently, $\eta(\bs{y}^*;\tau_{k,t}^2)$ is $(\alpha/\Delta+T)L_\theta^2$-Lipschitz.

    Moreover, let
    \begin{align}
        \langle A(\theta)\rangle_{\bs{y}^*_i}=\dfrac{\displaystyle\int A(\theta)\exp(-\dfrac{(\bs{y}^*_i-\theta)^2(\tau_{k,t}^{-2}+t)}{2})\pi(\theta
        )d\theta}{\displaystyle\int \exp(-\dfrac{(\bs{y}^*_i-\theta)^2(\tau_{k,t}^{-2}+t)}{2})\pi(\theta
        )d\theta}
    \end{align}
    then $|d\eta'(\bs{y}^*_i;\tau_{k,t}^2)/d\bs{y}^*_i|=|(\tau_{k,t}^{-2}+t)^2(\langle\theta^3\rangle_{\bs{y}^*_i}-\langle\theta^2\rangle_{\bs{y}^*_i}\langle\theta\rangle_{\bs{y}^*_i}-2\langle\theta\rangle_{\bs{y}^*_i}\langle\theta^2\rangle_{\bs{y}^*_i}+2\langle\theta\rangle_{\bs{y}^*_i}^3)|\leq 6(\alpha/\Delta+T)^2L_\theta^3$ which means that $\eta'(\bs{y}^*_i;\tau_{k,t}^2)$ is $6(\alpha/\Delta+T)^2L_\theta^3$-Lipschitz. Therefore,
    \begin{align}
        &|\dfrac{1}{N\alpha}\sum_{i=1}^N[\eta'(\bs{y}^*;\tau_{k,t}^2)]_i-\dfrac{1}{N\alpha}\sum_{i=1}^N[\eta'(\tilde{\bs{y}};\tau_{k,t}^2)]_i|\\
        =&|\dfrac{1}{N\alpha}\sum_{i=1}^N\eta'(\bs{y}^*_i;\tau_{k,t}^2)-\dfrac{1}{N\alpha}\sum_{i=1}^N\eta'(\tilde{\bs{y}}_i;\tau_{k,t}^2)|\\
        \leq&\dfrac{1}{N\alpha}\sum_{i=1}^N|\eta'(\bs{y}^*_i;\tau_{k,t}^2)-\eta'(\tilde{\bs{y}}_i;\tau_{k,t}^2)|\\
        \leq&\dfrac{1}{N\alpha}\sum_{i=1}^N6(\alpha/\Delta+T)^2L_\theta^3|\bs{y}^*_i-\tilde{\bs{y}}_i| \\
        \leq&\dfrac{6(\alpha/\Delta+T)^2L_\theta^3}{\sqrt{N}\alpha}\|\bs{y}^*-\tilde{\bs{y}}\|_2
    \end{align}
    which means that $\dfrac{1}{N\alpha}\sum_{i=1}^N[\eta'(\bs{y}^*;\tau_{k,t}^2)]_i$ is $\dfrac{6(\alpha/\Delta+T)^2L_\theta^3}{\sqrt{N}\alpha}$-Lipschitz. 
    
    Denote $C_b:=\dfrac{6(\alpha/\Delta+T)^2L_\theta^3}{\sqrt{N}\alpha}$ and let $b_t^{(k+1)}(\bs{z},t):=\dfrac{1}{N\alpha}\sum_{i=1}^N[\eta'(\phi^T\bs{r}^{(k)}(\bs{z}, t)+\hat{m}^{(k)}(\bs{z},t);\tau_{k,t}^2)]_i$.
    
    Now we use induction to prove the following inequalities:
    \begin{align}
        \|\hat{m}^{(k)}(\bs{z}_1, t)-\hat{m}^{(k)}(\bs{z}_2, t)\|_2&\leq C(k, T)\|\bs{z}_1-\bs{z}_2\|_2 \label{lip_m_r_1}\\
        \|\bs{r}^{(k)}(\bs{z}_1, t)-\bs{r}^{(k)}(\bs{z}_2, t)\|_2&\leq C^\star(k, T)\|\bs{z}_1-\bs{z}_2\|_2 \label{lip_m_r_2} \\
        \|\bs{r}^{(k)}(\bs{z}_1, t)\prod_{i=0}^v b_t^{(k+i)}(\bs{z}_1,t)-\bs{r}^{(k)}(\bs{z}_2, t)\prod_{i=0}^v b_t^{(k+i)}(\bs{z}_2,t)\|_2&\leq C'(k, v, T)\|\bs{z}_1-\bs{z}_2\|_2 \text{ for all } v\geq 0\label{lip_m_r_3}
    \end{align}
    where 
    \begin{align}
        C(k+1, T)=&(\alpha/\Delta+T)L_\theta^2(C_{\phi}C^\star(k, T)+C(k, T)) \\
        C^*(k+1, T)=&(C_{\phi}C(k+1, T, N)+C'(k, 0, T)) \\
        C'(k+1, v, T)=&\{\|y\|_2  (\frac{2L_\theta^2}{\alpha})^v \frac{6(\alpha/\Delta+T)^2L_\theta^3}{\alpha}+C_{\phi}[(v-1)(\frac{L_\theta^2}{\alpha})^{v-1}\dfrac{6(\alpha/\Delta+T)^2L_\theta^4}{\alpha}\notag \\
        &+(\frac{L_\theta^2}{\alpha})^{v}C(k+1,T)]+C'(k,v+1,T)\}.
    \end{align}

    Suppose \eqref{lip_m_r_1}, \eqref{lip_m_r_2} and \eqref{lip_m_r_3} holds for $l=0,\cdots, k$, where $1\leq k\leq K-1$.

    For $l=k+1$, we have
    \begin{align}
        &\|\hat{m}^{(k+1)}(\bs{z}_1, t)-\hat{m}^{(k+1)}(\bs{z}_2, t)\|_2 \\
        =&\|\eta(\phi^T\bs{r}^{(k)}(\bs{z}_1, t)+\hat{m}^{(k)}(\bs{z}_1,t); \tau_{k,t}^2)-\eta(\phi^T\bs{r}^{(k)}(\bs{z}_2, t)+\hat{m}^{(k)}(\bs{z}_2,t); \tau_{k,t}^2)\|_2\\
        \leq&(\alpha/\Delta+T)L_\theta^2\|\phi^T(\bs{r}^{(k)}(\bs{z}_1, t)-\bs{r}^{(k)}(\bs{z}_2, t))+\hat{m}^{(k)}(\bs{z}_1,t)-\hat{m}^{(k)}(\bs{z}_2,t)\|_2 \\
        \leq&(\alpha/\Delta+T)L_\theta^2(\|\phi\|_{\text{op}}C^\star(k, T)+C(k, T))\|\bs{z}_1-\bs{z}_2\|_2\\
        \leq&(\alpha/\Delta+T)L_\theta^2(C_{\phi}C^\star(k, T)+C(k, T))\|\bs{z}_1-\bs{z}_2\|_2\\
        =& C(k+1, T)\|\bs{z}_1-\bs{z}_2\|_2
    \end{align}
    and
    \begin{align}
        &\|\bs{r}^{(k+1)}(\bs{z}_1, t)-\bs{r}^{(k+1)}(\bs{z}_2, t)\| \\
        =&\|-\phi(\hat{m}^{(k+1)}(\bs{z}_1, t)-\hat{m}^{(k+1)}(\bs{z}_2, t))+\bs{r}^{(k)}(\bs{z}_1, t)b_t^{(k)}(\bs{z}_1, t)-\bs{r}^{(k)}(\bs{z}_2, t)b_t^{(k)}(\bs{z}_2, t)\|_2\\
        \leq&(\|\phi\|_{\text{op}}C(k+1, T)+C'(k, 0, T))\|\bs{z}_1-\bs{z}_2\|_2\\
        \leq & (C_{\phi}C(k+1, T)+C'(k, 0, T))\|\bs{z}_1-\bs{z}_2\|_2 \\
        =& C^*(k+1, T)\|\bs{z}_1-\bs{z}_2\|_2
    \end{align}
    and
    \begin{align}
        &\|\bs{r}^{(k+1)}(\bs{z}_1, t)\prod_{i=0}^v b_t^{(k+1+i)}(\bs{z}_1,t)-\bs{r}^{(k+1)}(\bs{z}_2, t)\prod_{i=0}^v b_t^{(k+1+i)}(\bs{z}_2,t)\|_2\\
        =&\|\bs{y}(\prod_{i=0}^v b_t^{(k+1+i)}(\bs{z}_1,t)-\prod_{i=0}^v b_t^{(k+1+i)}(\bs{z}_2,t))\notag\\
        &-\phi(\hat{m}^{(k+1)}(\bs{z}_1, t)\prod_{i=0}^v b_t^{(k+1+i)}(\bs{z}_1,t)-\hat{m}^{(k+1)}(\bs{z}_2, t)\prod_{i=0}^v b_t^{(k+1+i)}(\bs{z}_2,t))\notag\\
        &+\bs{r}^{(k)}(\bs{z}_1, t)\prod_{i=0}^{v+1} b_t^{(k+i)}(\bs{z}_1,t)-\bs{r}^{(k)}(\bs{z}_2, t)\prod_{i=0}^{v+1} b_t^{(k+i)}(\bs{z}_2,t)\|_2\\
        \leq& \{\|\bs{y}\|_2  (v+1)(\dfrac{L_\theta^2}{\alpha})^vC_b+\|\phi\|_{\text{op}} )\|\bs{z}_1-\bs{z}_2\|_2\\
        \leq & \{\|\bs{y}\|_2  (\frac{2L_\theta^2}{\alpha})^v C_b+C_{\phi}[(v-1)(\frac{L_\theta^2}{\alpha})^{v-1}C_b\sqrt{N}L_\theta+(\frac{L_\theta^2}{\alpha})^{v}C(k+1,T)] \notag\\
        &+C'(k,v+1,T)\}\|\bs{z}_1-\bs{z}_2\|_2 \\  
        = & \{\|\bs{y}\|_2  (\frac{2L_\theta^2}{\alpha})^v \frac{6(\alpha/\Delta+T)^2L_\theta^3}{\alpha}+C_{\phi}[(v-1)(\frac{L_\theta^2}{\alpha})^{v-1}\dfrac{6(\alpha/\Delta+T)^2L_\theta^4}{\alpha}+(\frac{L_\theta^2}{\alpha})^{v}C(k+1,T)] \notag\\
        &+C'(k,v+1,T)\}\|\bs{z}_1-\bs{z}_2\|_2 \\ 
        =&C'(k+1, v, T)\|\bs{z}_1-\bs{z}_2\|_2.
    \end{align}

    Finally, we prove \eqref{lip_m_r_1}, \eqref{lip_m_r_2} and \eqref{lip_m_r_3} holds for $k=0$.

    Obviously, \eqref{lip_m_r_1}, \eqref{lip_m_r_2} hold since $\hat{m}^{(0)}(\bs{z}_1, t)=\hat{m}^{(0)}(\bs{z}_1, t)=0$ and $\bs{r}^{(0)}(\bs{z}_1,t)-\bs{r}^{(0)}(\bs{z}_2,t)=y$.

    Moreover, 
    \begin{align}
        &\|\bs{r}^{(0)}(\bs{z}_1, t)\prod_{i=0}^v b_t^{(i)}(\bs{z}_1,t)-\bs{r}^{(0)}(\bs{z}_2, t)\prod_{i=0}^v b_t^{(i)}(\bs{z}_2,t)\|_2 \\
        =&\|y\|_2\|\prod_{i=0}^v b_t^{(i)}(\bs{z}_1,t)-\prod_{i=0}^v b_t^{(i)}(\bs{z}_2,t)\|_2\\
        \leq&\|y\|_2(2\dfrac{L_\theta^2}{\alpha})^vC_b \\
        \leq &\|y\|_2(2\dfrac{L_\theta^2}{\alpha})^v  \frac{6(\alpha/\Delta+T)^2L_\theta^3}{\alpha}.
    \end{align}
    We only need to let $C'(0, v, T)=\|y\|_2(2\dfrac{L_\theta^2}{\alpha})^v  \dfrac{6(\alpha/\Delta+T)^2L_\theta^3}{\alpha}$ to complete the proof.
 
\end{proof}

\begin{proof}[Proof of Theorem \ref{theorem 7}]
    \begin{align}
        &\dfrac{1}{N}D_{\text{KL}}(P_{\bs{z}_T/T}\|P_{\tilde{\bs{z}}_T/T}) \\
        \lesssim
        & \sum_{l=1}^L\E\int_{t_{l-1}}^{t_l}\dfrac{1}{N}\|m(\bs{z}_t, t)-\hat{m}(\bs{z}_{t_l}, t_l)\|_2^2dt \label{Girsanov_thm}\\
        \lesssim
        & \sum_{l=1}^L\E\int_{t_{l-1}}^{t_l}\dfrac{1}{N}\left( \|m(\bs{z}_t, t)-\hat{m}(\bs{z}_t, t)\|_2^2 + \|\hat{m}(\bs{z}_t, t)-\hat{m}(\bs{z}_{t_l}, t)\|_2^2 +\|\hat{m}(\bs{z}_{t_l}, t)-\hat{m}(\bs{z}_{t_l}, t_l)\|_2^2 \label{KL_three_terms}  \right)dt
    \end{align}

    where \eqref{Girsanov_thm} follows by Girsanov's theorem (see for example \cite[Proposition D.1]{oko2023diffusion}).

    For the first term of \eqref{KL_three_terms}, 
    \begin{align}
        &\lim_{K\rightarrow\infty}\lim_{N\rightarrow\infty}\sum_{l=1}^L\E\int_{t_{l-1}}^{t_l}\dfrac{1}{N} \|m(\bs{z}_t, t)-\hat{m}(\bs{z}_t, t)\|_2^2dt \\
        =
        &\lim_{K\rightarrow\infty}\sum_{l=1}^L\int_{t_{l-1}}^{t_l}\lim_{N\rightarrow\infty}\E\dfrac{1}{N} \|m(\bs{z}_t, t)-\hat{m}(\bs{z}_t, t)\|_2^2dt\label{first_eq1} \\
        =
        &0\label{first_eq3}
    \end{align} 
    where \eqref{first_eq1} follows from Dominated Convergence Theorem, and \eqref{first_eq3} follows from Theorem \ref{thm1}. Therefore, we could choose $K$ such that $1/N\sum_{l=1}^L\E\int_{t_{l-1}}^{t_l} \|m(\bs{z}_t, t)-\hat{m}(\bs{z}_t, t)\|_2^2dt\leq\epsilon/2$ for large $N$. For such $K$, $\delta$ could be given as follows:
    
    For the second term of \eqref{KL_three_terms}, by Lemma \ref{lemma10}, with probability $1-o_N(1)$, we have
    \begin{align}
        &\dfrac{1}{N}\E\|\hat{m}(\bs{z}_t, t)-\hat{m}(\bs{z}_{t_l}, t)\|_2^2 \\
        \leq
        & \dfrac{1}{N}C_4^2\E\|\bs{z}_t-\bs{z}_{t_l}\|_2^2 \\
        =
        &\dfrac{1}{N}C_4^2\E\|\int_{t_{l-1}}^tm(\bs{z}_s, s)ds+W_{t}-W_{t_l}\|_2^2 \\
        \leq
        & \dfrac{2}{N}C_4^2\left(\E\|\int_{t_{l-1}}^tm(\bs{z}_s, s)ds\|_2^2+\E\|W_{t}-W_{t_l}\|_2^2\right) \\
        \leq
        & 2C_4^2(L_\theta^2(t-t_{l-1})^2+t-t_{l-1}). \\
    \end{align}
    Therefore, 
    \begin{align}
        &\dfrac{1}{N}\sum_{l=1}^L\E\int_{t_{l-1}}^{t_l}\|\hat{m}(\bs{z}_t, t)-\hat{m}(\bs{z}_{t_l}, t)\|_2^2dt \\
        \leq
        & 2C_4^2\sum_{l=1}^L\int_{t_{l-1}}^{t_l}(L_\theta^2(t-t_{l-1})^2+t-t_{l-1})dt \\
        =
        &  2C_4^2 \sum_{l=1}^L(\dfrac{L_\theta^2\delta^3}{3}+\dfrac{\delta^2}{2}) \\
        =
        & 2C_4^2(\dfrac{L_\theta^2T\delta}{3}+\dfrac{T}{2})\delta .\label{sec_term}
    \end{align}

    For the third term of \eqref{KL_three_terms}, by Lemma \ref{lemma11}, with probability $1-o_N(1)$, we have
    \begin{align}
        &\sum_{l=1}^L\int_{t_{l-1}}^{t_l}\dfrac{1}{N}\|\hat{m}(\bs{z}_{t_l}, t)-\hat{m}(\bs{z}_{t_l}, t_l)\|_2^2dt \\
        \leq
        & \sum_{l=1}^L\int_{t_{l-1}}^{t_l} C_3^2(t-t_{l-1})^2dt \\
        \leq
        & \sum_{l=1}^L C_3^2\delta^3/3 \\
        =
        & \dfrac{C_3^2T\delta^2}{3}. \label{thi_term}
    \end{align}

    By \eqref{sec_term} and \eqref{thi_term}, \eqref{theorem 7} holds if we let 
    \begin{align}
        \delta\leq \min\left\{1, \dfrac{\epsilon/2}{2C_4^2(\dfrac{L_\theta^2T}{3}+\dfrac{T}{2})+ \dfrac{C_3^2T}{3}} \right\}.
    \end{align}
\end{proof}

Meanwhile, we also prove another condition for the theorem to hold, which we stated before as Corollary~\ref{large alpha case}. This can be derived directly from the following lemma.

\begin{lemma}
    If the derivative of $\Sigma^2\mapsto {\rm mmse}^*(\Sigma^{-2})$ is continuous at $\Sigma^2=0$, the fixed point equation \eqref{fixed point} will always have a unique solution for any fixed $\Delta$ when $\alpha$ is sufficiently large.
\end{lemma}

\begin{proof}
As $\alpha\to\infty$, 
any fixed point of \eqref{fixed point}, 
$E$, will converge to $0$. 
Suppose there are two different fixed points, 
$E_1 < E_2$.
Then, there exists one $E_3\in[E1,E2]$ such that 
\begin{equation}\label{dfixed point}
    1=\frac{d}{dE}{\rm mmse}^*(\frac{\alpha}{\Delta+E})|_{E=E_3},
\end{equation}
which is the derivative of the fixed point equation. Thus, by the regularity assumption, the derivative of ${\rm mmse}^*$ can be bounded near 0. That is, we can find $\delta>0$,
$\frac{d}{d\Sigma^2}{\rm mmse}^*(\Sigma^{-2})<\delta$ holds for all $\Sigma$ near 0.

Thus, as $E_3\to 0$ when $\alpha\to \infty$, the right hand side of \eqref{dfixed point} can be bounded as follows:
\begin{align}
    &\frac{d}{dE}\left.{\rm mmse}^*(\frac{\alpha}{\Delta+E})\right|_{E=E_3}
    \\
    =&\frac{d}{d(\frac{\Delta+E}{\alpha})}{\rm mmse}^*(\frac{\alpha}{\Delta+E})\frac{d}{dE}\left.\left(\frac{\Delta+E}{\alpha}\right)\right|_{E=E_3}\\
\le&\frac{\delta}{\alpha}\to 0,
\end{align}
which leads to a contradiction, as the L.H.S. of \eqref{dfixed point} equals to 1.
\end{proof}

\section{BOUNDING THE WASSERSTEIN DISTANCE BETWEEN $z_T/T$ AND $\tilde{z}_T/T$}
\label{sec_w}

Here we improve our main result Theorem \ref{theorem 7} by giving the bound of $1/\sqrt{N}W_2(P_{z_T/T},P_{\tilde{z}_T/T})$. We use two terms to bound it by applying triangle inequality. The first term is $1/\sqrt{N}W_2(P_{\hat{z}_T/T},P_{\tilde{z}_T/T})$ where $\hat{z}_T$ and $\tilde{z}_T$ are generated by continuous AMP process\eqref{process2} and discretized AMP process\eqref{process3}. We show that the gap between these two AMP processes is small by using the properties that the maps $\bs{z}_t \mapsto\hat{m}(\bs{z}_t, t)$ and $t \mapsto\hat{m}(\bs{z}_t, t)$ are Lipschitz continuous shown in Lemma \ref{lemma11} and Lemma \ref{lemma10}.

\begin{lemma}\label{lemma_hat_tilde_z}
If $\delta<1$ and \eqref{mhat_part1}, \eqref{mhat_part2} and \eqref{BM_prop2} hold, the gap between two AMP processes is small, i.e.
    \begin{align}
        \sum_{l=1}^L\dfrac{1}{N}\int_{t_{l-1}}^{t_l}\|\hat{m}(\hat{z}_t,t)-\hat{m}(\tilde{z}_{t_{l-1}}, t_{l-1})\|_2^2dt\rightarrow 0 \text{ \ as \ }\delta\rightarrow 0.
        \label{mhat_contin}
    \end{align}
\end{lemma}

\begin{proof}
    Let $B_l=\dfrac{1}{N}\displaystyle\int_{t_{l-1}}^{t_l}\|\hat{m}(\hat{z}_t,t)-\hat{m}(\tilde{z}_{t_{l-1}}, t_{l-1})\|_2^2dt$ and $D_l=\|\hat{z}_{t_l}-\tilde{z}_{t_l}\|_2$ for $l=1,2,\cdots, L$
    We claim that there exist constants $C_5(T,N,K,\epsilon)$ and $C_6(T,N,K,\epsilon)$ such that
    \begin{align}
        B_l\leq C_5\delta^{2}  (C_6\delta+1)^{l-1} \text{ for }  l=1,2,...L.
        \label{Bl_induction}
    \end{align}
    Then  $\displaystyle\sum_{l=1}^LB_l\leq\dfrac{C_5}{C_6}\delta((C_6\delta+1)^L-1)\leq \delta\dfrac{C_5}{C_6}e^{C_6T}$ which converges to 0 as $\delta\rightarrow 0$.

    Now we use induction to prove \eqref{Bl_induction}. Suppose \eqref{Bl_induction} holds for $i=1,2,...,l-1$. For $B_l$, it can be calculated as:
    \begin{align}
        B_l&\leq \dfrac{1}{N}\displaystyle\int_{t_{l-1}}^{t_l}(2\|\hat{m}(\hat{z}_t,t)-\hat{m}(\hat{z}_{t}, t_{l-1})\|_2^2+2\|\hat{m}(\hat{z}_{t}, t_{l-1})-\hat{m}(\tilde{z}_{t_{l-1}}, t_{l-1})\|_2^2)dt \\
        &\leq\dfrac{2C_3^2}{3}\delta^3 +  \dfrac{2C_4^2}{N}\displaystyle\int_{t_{l-1}}^{t_l}\|\hat{z}_t-\tilde{z}_{t_{l-1}}\|_2^2dt\\
        &\leq \dfrac{2C_3^2}{3}\delta^3+\dfrac{4C_4^2}{N}\displaystyle\int_{t_{l-1}}^{t_l}\|\hat{z}_t-\hat{z}_{t_{l-1}}\|_2^2dt+\dfrac{4C_4^2}{N}\displaystyle\int_{t_{l-1}}^{t_l}D_{l-1}^2dt,
        \label{first_term}
    \end{align}
    where the first inequality is from \eqref{mhat_part1} and \eqref{mhat_part2}. Then, consider the second term in the last equation, by \eqref{BM_prop2}, we have:
    \begin{align}
        \int_{t_{l-1}}^{t_l}\|\hat{z}_t-\hat{z}_{t_{l-1}}\|_2^2=&\int_{t_{l-1}}^{t_l}\|\int_{t_{l-1}}^t\hat{m}(\hat{z}_s,s)ds+dW_s\|_2^2dt\\
        \leq&\int_{t_{l-1}}^{t_l}2(\|\int_{t_{l-1}}^t\hat{m}(\hat{z}_s, s)ds\|_2^2+\|W_t-W_{t_{l-1}}\|_2^2)dt\\
        \leq&\int_{t_{l-1}}^{t_l}(2L_{\theta}^2\delta^2+2\|W_t-W_{t_{l-1}}\|_2^2)dt\\
        \leq& 2L_{\theta}^2\delta^3+2NC_2\delta^2.
        \label{second_term}
    \end{align}
    Now, we consider the third term.
    \begin{align}
        D_l-D_{l-1}&\leq \|\hat{z}_{t_l}-\hat{z}_{t_{l-1}}-(\tilde{z_{t_l}}-\tilde{z}_{t_{l-1}})\|_2 \\
        &\leq \displaystyle\int_{t_{l-1}}^{t_l}\|\hat{m}(\hat{z}_t, t)-\hat{m}(\tilde{z}_{t_{l-1}}, t_{l-1})\|_2dt \\
        &\leq (\delta N B_l)^{1/2}.
    \end{align}
    Since $D_0=0$, we have 
    \begin{align}
    D_l\leq \sqrt{\delta N}\displaystyle\sum_{s=1}^{l}B_s^{1/2}\leq \sqrt{\delta N l}(\displaystyle\sum_{s=1}^{l}B_s)^{1/2}\leq\sqrt{NT}(\displaystyle\sum_{s=1}^{l}B_s)^{1/2}
    \label{third_term}
    \end{align}
    Combining \eqref{first_term}, \eqref{second_term} and \eqref{third_term}, we have
    \begin{align}
        B_l \leq& \dfrac{2C_3^2}{3}\delta^3+
        \dfrac{4C_4^2}{N}(2L_{\theta}^2\delta^3+2NC_2\delta^2)+
        \dfrac{4C_4^2}{N}\displaystyle\int_{t_{l-1}}^{t_l}NT(\displaystyle\sum_{s=1}^{l-1}B_s) \\
        \leq&  [\dfrac{2C_3^2}{3}+
        \dfrac{4C_4^2}{N}(2L_{\theta}^2+2NC_2)]\delta^2+
        4C_4^2T\delta(\displaystyle\sum_{s=1}^{l-1}B_s).
        \label{33}
    \end{align}
    Let $C_5=\dfrac{2C_3^2}{3}+\dfrac{4C_4^2}{N}(2L_{\theta}^2+2NC_2)$ and $C_6 = 4TC_4^2$, by \eqref{Bl_induction}, we have
    \begin{align}
        B_l\leq& C_5\delta^2+C_6\delta(\sum_{s=1}^{l-1}B_s)\\
        \leq&C_5\delta^2+C_6\delta[\sum_{s=1}^{l-1}C_5\delta^{2}  (C_6\delta+1)^{l-1}]\\
        =&C_5\delta^{2}  (C_6\delta+1)^{l}.
    \end{align}
    
    Moreover, we need to prove the case when $l=1$. Actually, in \eqref{first_term}, the third term equals to $0$ when $l=1$. Thus, by \eqref{33}, we have,
    \begin{align}
        B_1\leq (2L_{\theta}^2+2NC_2)]\delta^{2}=C_5\delta^{2}.
    \end{align}
    
\end{proof}

The second term is $1/\sqrt{N}W_2(P_{z_T/T},P_{\hat{z}_T/T})$ where $z_T$ and $\hat{z}_T$ are generated by diffusion process\eqref{process1} with drift term $m(z_t,t)$ and continuous AMP process\eqref{process2}. We show that the second term will converge to $0$ under the assumption that 
$z_t \mapsto m(z_t, t)$ is Lipschitz continuous. Finally, we give the following theorem.

\begin{theorem}
    If $m(z_t, t)$ is $L_m$-Lipschitz w.r.t $z_t$, 
    \begin{align}
        \lim_{\delta\rightarrow 0}\lim_{K\rightarrow\infty}\lim_{N\rightarrow\infty}\dfrac{1}{\sqrt{N}}W_2(P_{z_T/T},P_{\tilde{z}_T/T}) =0.
    \end{align}
\end{theorem}

\begin{proof}
    \begin{align}
        &\lim_{\delta\rightarrow 0}\lim_{K\rightarrow\infty}\lim_{N\rightarrow\infty}\dfrac{1}{\sqrt{N}}W_2(P_{z_T/T},P_{\tilde{z}_T/T})\\ 
        \leq & 
        \lim_{\delta\rightarrow 0}\lim_{K\rightarrow\infty}\lim_{N\rightarrow\infty}\dfrac{1}{\sqrt{N}}W_2(P_{\hat{z}_T/T},P_{\tilde{z}_T/T}) +\lim_{K\rightarrow\infty}\lim_{N\rightarrow\infty}\dfrac{1}{\sqrt{N}}W_2(P_{z_T/T},P_{\hat{z}_T/T}).\label{W2_first}
    \end{align}

    For the first term of \eqref{W2_first}, we have
    \begin{align}
        &\dfrac{1}{\sqrt{N}}W_2(P_{\hat{z}_T/T},P_{\tilde{z}_T/T}) \\
        \leq & 
        \dfrac{1}{\sqrt{N}T}(\E\|\tilde{z}_T-\hat{z}_T\|_2^2)^{1/2} \\
        \leq & 
        \dfrac{1}{\sqrt{N}T}(\E\|\sum_{l=1}^{L}\int_{t_{l-1}}^{t_l}(\hat{m}(\hat{z}_t,t)-\hat{m}(\tilde{z}_{t_{l-1}}, t_{l-1}))dt\|_2^2)^{1/2} \\
        \leq &
        \dfrac{1}{\sqrt{N}}(\E\sum_{l=1}^{L}\int_{t_{l-1}}^{t_l}\|\hat{m}(\hat{z}_t,t)-\hat{m}(\tilde{z}_{t_{l-1}}, t_{l-1})\|_2^2dt)^{1/2}
    \end{align}

    which converges to $0$ as $\delta\rightarrow 0$ by applying Lemma \ref{lemma_hat_tilde_z}.

    For the second term of \eqref{W2_first}, we first consider $1/N\E\|z_T-\hat{z}_T\|_2^2$
    \begin{align}
        &\dfrac{1}{N}\E\|z_T-\hat{z}_T\|_2^2 \\
        =&
        \dfrac{1}{N}\E\|\int_0^T(m(z_t, t)-\hat{m}(\hat{z}_t, t))dt\|_2^2\\
        \leq&
        \dfrac{1}{N}\E\int_0^T\|m(z_t, t)-\hat{m}(\hat{z}_t, t)\|_2^2dt \\
        \lesssim&
        \dfrac{1}{N}\E\int_0^T\|m(z_t, t)-m(\hat{z}_t, t)\|_2^2dt+
        \dfrac{1}{N}\E\int_0^T\|m(\hat{z}_t, t)-\hat{m}(\hat{z}_t, t)\|_2^2dt\\
        \leq&
        \dfrac{1}{N}\E\int_0^TL_m^2\|z_t-\hat{z}_t\|_2^2dt+o_{K,N}
    \end{align}
    where $\lim_{K\rightarrow\infty}\lim_{N\rightarrow\infty}o_{K,N}\rightarrow 0$ by using Theorem \ref{thm1}.
    
    Let $f(t)=1/N\E\|z_t-\hat{z}_t\|_2^2$, then we get $f(t)\lesssim L_m^2\int_0^Tf(t)dt+o_{K,N}$ which implies $f(t)\lesssim o_{K,N}\exp{(L_m^2t)}$.

    Therefore,
    \begin{align}
        &\lim_{K\rightarrow\infty}\lim_{N\rightarrow\infty}\dfrac{1}{N}W_2^2(P_{z_T/T},P_{\hat{z}_T/T}) \\
        \leq&
        \lim_{K\rightarrow\infty}\lim_{N\rightarrow\infty}\dfrac{1}{NT^2}\E\|z_T-\hat{z}_T\|_2^2\\
        \lesssim&
        \lim_{K\rightarrow\infty}\lim_{N\rightarrow\infty}\dfrac{1}{T^2}(o_{K,N}e^{L_m^2T})\\   
        =&0.
    \end{align}
\end{proof}

\subsection{Connection between overlap concentration and Lipschitz constant of $m_t$}\label{sec_overlap}
In Appendix~\ref{sec_w}, we showed that convergence under the Wasserstein distance can be established assuming that $m_t(z_t):=m(z_t,t)$ has a dimension free Lipschitz constant.
In this section, we point out a connection between such a condition and a concentration property of the overlap.

For the Gibbs measure, a fundamental quantity is the overlap:
\begin{align}
R_{1,2}:=\frac1{p}\theta^{1\top}\theta^2
\end{align}
where $\theta^1$ and $\theta^2$ are two independent random vectors (spins) drawn from the Gibbs measure.
we write
$q_p:=\left<R_{1,2}\right>$, where $\left<\cdot\right>$ denotes the average with respect to the Gibbs measure
(note that \cite[(1.74)]{talagrand2010mean} defined $q$, which is the a.s.\ limit of $q_p$).
For the SK model, it is well known that $R_{1,2}$ concentrates around $q$ only in a certain parameter regime, 
whose boundary is conjectured to be the Almeida-Thouless (AT) line \citep[Section~1.8]{talagrand2010mean}.

In this section, we argue that if $R_{1,2}$ does not concentrate, then $m_t$ cannot have a dimension-free Lipschitz constant:
\begin{proposition}
Suppose that the prior of each coordinate of $\theta$ is supported on $[-1,1]$. 
Let
$\lambda_{\rm m}:=\lambda_{\max}(\nabla m_t)
=
\lambda_{\max}(\left<\theta\theta^{\top}\right>
-\left<\theta\right>\left<\theta\right>^{\top})$, then
\begin{align}
\frac{p}{2}\left<|R_{1,2}-q_p|^2\right>
\le 
\lambda_{\rm m}
\le 
p\sqrt{\left<|R_{1,2}-q_p|^2\right>}.
\label{e_126}
\end{align}
\end{proposition}
As a consequence, $\lambda_{\rm m}=O(1)$ only if $\left<|R_{1,2}-q_p|^2\right>=O(1/p)$, which is the typical order of the variance of the overlap in the high-temperature regime \citep{talagrand2010mean}.
Conversely, if $\left<|R_{1,2}-q_p|^2\right>=O(1/p)$, then by \eqref{e_126}, $\lambda_{\rm m}\le \sqrt{p}$, which is sharper than the \it{a priori} estimate $\lambda_{\rm m}=O(p)$ using the information $\|\theta\|_2^2\le p$.
\begin{proof}
Note that 
\begin{align}
\left<|R_{1,2}-q_p|^2\right>
=\frac1{p^2}\left<|\theta^{1\top}\theta^2|^2\right>
-\frac1{p^2}
|\left<\theta\right>^{\top}
\left<\theta\right>|^2;
\end{align}
and
\begin{align}
\left<\theta^{\top}\left(
\left<\theta\theta^{\top}
\right>
-\left<\theta\right>\left<\theta\right>^{\top}
\right)
\theta\right>
=
\left<|\theta^{1\top}\theta^2|^2\right>
-
\tr\left(\left<\theta\right>
\left<\theta\right>^{\top}
\left<\theta\theta^{\top}\right>
\right);
\label{e126}
\end{align}
and moreover
\begin{align}
\left<\theta\right>^{\top}
\left(
\left<\theta\theta^{\top}
\right>
-\left<\theta\right>\left<\theta\right>^{\top}
\right)
\left<\theta\right>
=
\tr\left(\left<\theta\right>
\left<\theta\right>^{\top}
\left<\theta\theta^{\top}\right>
\right)
-
|\left<\theta\right>^{\top}
\left<\theta\right>|^2.
\label{e127}
\end{align}
Note that \eqref{e126} and \eqref{e127} are both nonnegative and their sum is $p^2\left<|R_{1,2}-q_p|^2\right>$, so one of them is at least $\frac{p^2}{2}\left<|R_{1,2}-q_p|^2\right>$. 
Since the norm of $\theta$ is at most $\sqrt{p}$, the lower bound on the eigenvalue follows.

Similarly, by Jensen's inequality we have $\left<\theta\right>^{\top}\left<\theta\theta^{\top}\right>
\left<\theta\right>
\ge(\left<\theta\right>^{\top}\left<\theta\right>)^2$,
hence
\begin{align}
\lambda_{\rm \max}^2(\left<\theta\theta^{\top}\right>
-\left<\theta\right>\left<\theta\right>^{\top}
)
\le \tr\left[\left(\left<\theta\theta^{\top}\right>
-\left<\theta\right>\left<\theta\right>^{\top}
\right)^2
\right]
\le 
p^2\left<|R_{1,2}-q_p|^2\right>.
\end{align}
\end{proof}

While the non-concentration of the overlap is indeed possible for the SK model in some parameter regimes, 
it still remains to address if (or when) does non-concentration of the overlap happens in the linear-inverse problem.
For ``good'' prior distributions such as Gaussian, we clearly have $\left<|R_{1,2}-q_p|^2\right>=O(1/p)$ and indeed $\lambda_{\rm max}(\nabla m_t)=O(1)$ as the posterior is simply a Gaussian distribution with covariance matrix smaller than that of the prior.
On the other hand, consider the example where the prior is supported on $\{\pm 1\}$, $\phi_{ij}\sim\mathcal{N}(0,\frac1{n})$, and $\delta=n/p\to\infty$. 
Then sine $|\theta_j|$ is constant, we have
\begin{align}
\nu(\theta)\propto
\exp\left(\frac1{\Delta}y^{\top}\phi \theta
-
\frac1{2\Delta}\theta^{\top} (\phi^{\top}\phi-I)\theta
+
z^{\top}\theta
\right).
\label{e219}
\end{align}
It is well-known that the law of $\sqrt{\delta}(\phi^{\top}\phi-I)$ converges (in eigen-spectrum) to that of the Wigner random matrix when $n/p$ is large, so that $\nu$ converges to the Gibbs measure of the SK model
\begin{align}
\nu_{SK}(\theta)
\propto
\exp\left(\frac{\beta}{\sqrt{p}}\sum_{i<j\le p}g_{ij}\theta_i\theta_j+\sum_{i\le p}h_i\theta_i\right)
\end{align}
where $\beta=\frac1{\Delta\sqrt{\delta}}$ and $h=z+\frac1{\Delta}y^{\top}\phi$.
This suggests that $m_t$ may not have a dimension independent Lipschitz constant, at least in some parameter regime and for some $(y,z)$ (where $y$ does not necessarily follow the linear model \eqref{model1}).
This is in sharp contrast with the fact that $m_{K,t}$ does have a dimension free (depending on $K$) Lipschitz constant.

In the special case of \eqref{e219} where $y$ does follow the linear model \eqref{model1} (the so-called ``planted'' or ``Bayes optimal'' setting), 
\cite{barbier2020mutual} used the Nishimori identity to establish certain overlap concentration results.
Furthermore, the recent breakthroughs 
\citep{brennecke2023operator}\citep{el2024bounds} proved that for the SK model,
the interior of the overlap concentration regime also satisfies $O(1)$ operator norm bound for the covariance matrix.
These results strongly suggest the correctness our hypothesis about an $O(1)$ Lipschitz constant of our $m_t$,
although a complete proof does not immediately follow, 
since \citep{brennecke2023operator}\citep{el2024bounds} relied on a TAP characterization of the free energy,
which has been conjectured but not rigorously proved for the linear inverse problem in the general setting \citep{qiu2023tap}.

\section{ADDITIONAL EXPERIMENTS}\label{appen:results_of_numeric}
In this section we present some results of applying Algorithm \ref{Alg1} to sample from the posterior distribution in model \eqref{model1}. The prior distribution and the parameter settings are described in Section \ref{sec:numer_exper}.

\begin{figure}[t]
    \includegraphics[width=0.95\textwidth]{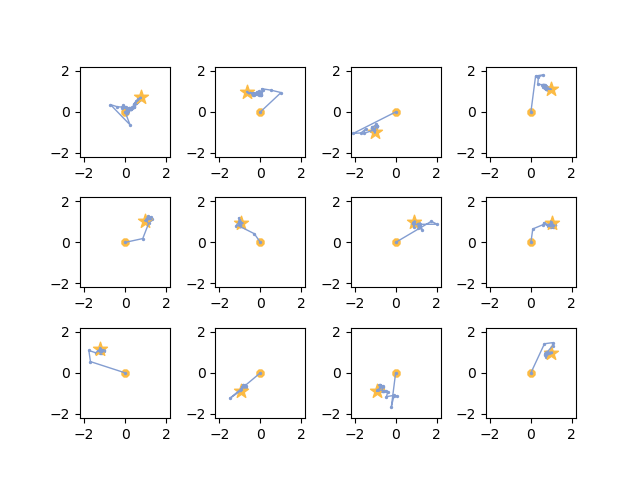}
    \caption{Trajectories of the first and second coordinates of the samples generated by Algorithm \ref{Alg1}. In this experiment, the prior distribution is $P_0(\theta)=\frac{1}{2}\delta_{-1}+\frac{1}{2}\delta_{1}$, and we set $M=1000, N=1250, K=20, T=200, \delta=0.1$. The three rows correspond to $\Delta$=0.01, 1, 10 respectively.}
    \label{fig:trajectory}
\end{figure}

Additionally, in Figure \ref{fig:alg_mse_dps_mse}, we calculate the {\rm Algorithm MSE} and {\rm MMSE} as follows:
\begin{align}
    {\rm Algorithm \ MSE} &=\frac{1}{2N} \|\bs{\theta}-\bs{\theta}^{\rm alg}\|_2^2\label{eq:alg_mse},\\
    {\rm MMSE} &= \frac{1}{N}\E [\|\bs{\theta}-\E[\bs{\theta}|\bs{\phi\theta}+\sqrt{\Delta/\alpha}\bs{w}]\|_2^2].
\end{align}
We use $1/(2N)$ rather than $1/N$ in \eqref{eq:alg_mse}, because if $\bs{\theta}^{\rm alg}$ is a perfect posterior sample, $\bs{\theta}$ and $\bs{\theta}^{\rm alg}$ are i.i.d.\ conditioned on $(\bs{X}, \bs{Y})$,
and hence $\mathbb{E}[\|\bs{\theta}-\bs{\theta}^{\rm alg}\|_2^2]= 2\mathbb{E}[\|\bs{\theta}-\mathbb{E}[\bs{\theta}|{\bf X,Y}]\|_2^2$.

For the second experiment, we also compare the sampling time between our algorithm and DPS-Gaussian method. From Table \ref{table:sample_time}, we observe that our algorithm is more efficient.
\begin{table}[h]
\caption{Sampling Time} \label{table:sample_time}
\begin{center}
\begin{tabular}{cccc}
\hline
\multicolumn{1}{l}{\textbf{Method}} & \textbf{N=192} & \textbf{N=768} & \textbf{N=1728} \\ \hline
\textbf{Our Algorithm}              & 8s             & 53s            & 390s            \\ \hline
\textbf{DPS-Gaussian}               & 184s           & 427s           & 882s            \\ \hline
\end{tabular}
\end{center}
\end{table}

\end{document}